\documentclass[10pt]{article}

\usepackage[english]{babel}
\usepackage{amsmath}
\usepackage{amssymb}
\usepackage{bm}
\usepackage{bbm}
\usepackage{mathrsfs,color}
\usepackage{graphics,graphicx,theorem}
\usepackage{dsfont}
\usepackage{setspace}

\usepackage{mathtools}

\usepackage{natbib}
\usepackage{multirow}
\usepackage{hhline}
\usepackage{hyperref}

\usepackage{authblk}

\hypersetup{
	colorlinks=true,
	urlcolor=blue,
	citecolor=blue,
	linktoc=all,
	linkcolor=red} 

\usepackage[scriptsize]{subfigure}
\allowdisplaybreaks

\makeatletter
\renewcommand\@biblabel[1]{}
\makeatother

\def\bSig\mathbf{\Sigma}

\newcommand{\R}{\mathds{R}}

\newcommand{\Pcr}{\mathscr{P}}

\newcommand{\Var}{\operatornamewithlimits{Var}}

\newcommand{\qed}{$\square$}

\newcommand{\D}{{\rm d}}

\newcommand{\Bias}{{\rm Bias}}

\newtheorem{thm}{Theorem}[section]
\newtheorem{prop}{Proposition}[section]
\newtheorem{remark}{Remark}[section]
\newtheorem{lemma}{Lemma}[section]

\newenvironment{proof}{\noindent \textit{Proof.}}{\hfill$\square$}

\usepackage[top=1.2in,bottom=1.2in,left=1.2in,right=1.2in]{geometry}

\begin{document}

\title{\bf {\Large{A Good-Turing estimator for feature allocation models}}}

\author[,1]{Fadhel Ayed \thanks{fadhel.ayed@gmail.com}}
\author[,2]{Marco Battiston \thanks{marco.battiston@stats.ox.ac.uk}}
\author[,3]{Federico Camerlenghi \thanks{federico.camerlenghi@unimib.it}}
\author[,4]{Stefano Favaro \thanks{stefano.favaro@unito.it}}

\affil[1]{University of Oxford}
\affil[2]{Lancaster University}
\affil[3]{University of Milano--Bicocca}
\affil[4]{University of Torino}

\date{}
\maketitle
\thispagestyle{empty}

\setcounter{page}{1}

\begin{abstract}
Feature allocation models generalize classical species sampling models by allowing every observation to belong to more than one species, now called features. Under the popular Bernoulli product model for feature allocation, we assume $n$ observable samples and we consider the problem of estimating the expected number $M_{n}$ of hitherto unseen features that would be observed if one additional individual was sampled. The interest in estimating $M_{n}$ is motivated by numerous applied problems where the sampling procedure is expensive, in terms of time and/or financial resources allocated, and further samples can be only motivated by the possibility of recording new unobserved features. We consider a nonparametric estimator $\hat{M}_{n}$ of $M_{n}$ which has the same analytic form of the popular Good-Turing estimator of the missing mass in the context of species sampling models. We show that $\hat{M}_{n}$ admits a natural interpretation both as a jackknife estimator and as a nonparametric empirical Bayes estimator. Furthermore, we give provable guarantees for the performance of $\hat{M}_{n}$ in terms of minimax rate optimality, and we provide with an interesting connection between $\hat{M}_{n}$ and the Good-Turing estimator for species sampling. Finally, we derive non-asymptotic confidence intervals for $\hat{M}_{n}$, which are easily computable and do not rely on any asymptotic approximation. Our approach is illustrated with synthetic data and SNP data from the ENCODE sequencing genome project. 
\end{abstract}


\maketitle

\section{Introduction} \label{sec1}

Feature allocation models generalize classical species sampling models by allowing every observation to belong to more than one species, now called features. In particular, every observation is endowed with a (unknown) finite set of features selected from a (possibly infinite) collection of features $(F_{j})_{j\geq1}$. We conveniently represent each observation as a binary sequence, with entries indicating the presence (1) or absence (0) of each feature. Every feature $F_{j}$ is associated to an unknown probability $p_{j}$, and each observation displays feature $F_{j}$ with probability $p_{j}$, for $j\geq1$. The Bernoulli product model, or binary independence model, is arguably the most popular feature allocation model. It models the $i$--th observation as a sequence $Y_{i}=(Y_{i,j})_{j\geq1}$ of independent Bernoulli random variables with unknown success probabilities $(p_{j})_{j\geq1}$, with the assumption that $Y_{r}$ is independent of $Y_{s}$ for any $r\neq s$. Bernoulli product models have found applications in several scientific disciplines. They have been applied extensively in ecology when animals are captured using traps and each observation is an incidence vector collecting the presence or absence of each species in the traps (e.g., \citet{Col(12)}, \citet{Cha(14)} and \citet{Cha(17)}). Besides ecology, Bernoulli product models found applications in the broad area of biosciences (e.g., \citet{Ion(09)}, \citet{Gra(14)} and  \citet{Zou(16)}); in the analysis of choice behaviour arising from psychology and marketing (\citet{Gor(06)}); in binary matrix factorization for dyadic data (\citet{Mee(07)}); in graphical models (e.g., \citet{Woo(06)} and \citet{Woo(07)}); in the analysis of similarity judgement matrices (\citet{Nav(07)}); in network data analysis (\citet{Mil(10)}).

Let $\mathbf{Y}_{n}=(Y_{1},\ldots,Y_{n})$ denote a collection of $n$ random samples collected under the Bernoulli product model with unknown feature probabilities $(p_{j})_{j\geq 1}$. Furthermore, let $X_{n,j}=\sum_{1\leq i\leq n}Y_{i,j}$ be the number of times that feature $F_{j}$ has been observed in $\mathbf{Y}_{n}$. That is, $X_{n,j}$ is a Binomial random variable with parameters $(n,p_{j})$ for any $j\geq1$. In this paper we consider the problem of estimating the conditional expected number, given the random sample $\mathbf{Y}_{n}$, of hitherto unseen features that would be observed if one additional sample $Y_{n+1}$ was collected, i.e.
\begin{equation}\label{missing_mass}
\begin{split}
M_{n}(\mathbf{Y}_{n},(p_{j})_{j\geq 1} )&=\mathbb{E}\left(\sum_{j\geq 1}\mathbb{I} {\{ X_{n, j} =0,Y_{n+1,j}=1\}}\,|\,\mathbf{Y}_{n}\right)\\
& =\sum_{j\geq 1} p_j \mathbb{I}{\{ X_{n, j} =0\}},
\end{split} 
\end{equation}
where $\mathbb{I}{}$ is the indicator function. For easiness of notation, in the rest of the paper we will not highlight the dependence on $\mathbf{Y}_{n}$ and $(p_{j})_{j\geq1}$, and simply write $M_{n}$ instead of $M_{n}(\mathbf{Y}_{n},(p_{j})_{j\geq 1} )$. The statistic $M_{n}$ is referred to as the missing mass, namely the sum of the probability masses of unobserved features in the first $n$ samples. 

Interest in estimating the missing mass is motivated by numerous applied problems where the sampling procedure is expensive, in terms of time and/or financial resources, and further draws can be only motivated by the possibility of recording new unobserved features. In genetics, for instance, the ambitious prospect of growing databases to encompass hundreds of thousands of human genomes, makes important  to quantify the power of large sequencing projects to discover new genetic variants (\citet{Aut(15)}). An accurate estimate of $M_{n}$ will enable better study design and quantitative evaluation of the potential and limitations of these datasets. Indeed one can fix a suitable threshold such that the sequencing procedure takes place until the estimate of $M_{n}$ becomes for the first time smaller than the threshold. This introduces a criterion for evaluating the effectiveness of further sampling, providing a roadmap for large-scale sequencing projects. A Bayesian parametric estimator of \eqref{missing_mass} has been introduced in \citet{Ion(09)}, and it relies on a Beta prior distribution for the unknown probabilities $p_{j}$'s. A limitation of this approach is that it requires parametric forms for the distribution of variant frequencies, which requires some model of demography and selection. For instance, the Beta prior distribution is a reasonable assumption for neutrally evolving variants but may not be appropriate for deleterious mutations. To overcome this drawback, a nonparametric approach to estimate  $M_{n}$ has been proposed in the recent work of \citet{Zou(16)}. This is a purely algorithmic type approach, whose core is a linear programming based algorithm.

In this paper, we consider an estimator $\hat{M}_{n}$ of $M_{n}$ which has the same analytic form of the popular Good-Turing estimator of the missing mass under the multinomial model for species sampling (e.g., \citet{Goo(53)} and \citet{Rob(68)}), with the difference that the two estimators have different ranges (supports). We refer to $\hat{M}_{n}$ as the Good-Turing estimator for feature allocation models. As in \citet{Zou(16)}, our approach is nonparametric in the sense that it does not rely on any distributional assumption on the unknown $p_{j}$'s. The use of the Good-Turing estimator for estimating $M_{n}$ was first considered in \citet{Cha(17)}, where also a bootstrap procedure is presented to approximate the variance of the estimator. Our work delves into the Good-Turing estimator for feature allocation models, thus providing theoretical guarantees for its use. We show that $\hat{M}_{n}$ admits a natural interpretation both as a jackknife (resampling) estimator (\citet{Que(56)} and \citet{Tuk(58)}) and as a nonparametric empirical Bayes estimator in the sense of \citet{Efr(73)}. Theoretical properties of the estimator $\hat{M}_{n}$ are investigated. Specifically, we first provide a lower bound for the minimax risk under a squared loss function and then we show that the mean squared error of the proposed estimator achieves the optimal minimax rate for the estimation of $M_{n}$. This asymptotic analysis provides with an interesting connection between the Good-Turing estimators for species sampling models and for feature allocation models in terms of the limiting Poisson regime of the Binomial distribution. Finally, we derive a non-asymptotic confidence interval for $\hat{M}_{n}$ by means of novel Bernstein type concentration inequalities which are of separate interest; the confidence interval is easy to compute and it does not rely on any asymptotic assumption or any parametric constraint on the unknown $p_{j}$'s. We illustrate the proposed methodology by the analysis of various synthetic data and SNP datasets from the ENCODE (http://www.hapmap.org/downloads/encode1.html.en) sequencing genome project.

The paper is structured as follows. In Section \ref{sec2} we introduce the Good-Turing estimator $\hat{M}_{n}$ for the missing mass $M_{n}$, we present its interpretation as  a jackknife estimator and a nonparametric empirical Bayes estimator, and we give provable guarantees for its performance. In Subsection \ref{ConfidenceInt} we derive a non-asymptotic confidence interval for $\hat{M}_{n}$, in Subsection \ref{StoppRule} we discuss the use of $\hat{M}_{n}$ for designing cost-effective feature inventories, and in Section \ref{NumIll} we present a numerical illustration of $\hat{M}_{n}$ with synthetic and real data. Some concluding remarks on future works are discussed in Section \ref{Conclusion}. Proofs are deferred to  Appendix \ref{app}.


\section{A Good-Turing estimator for $M_{n}$} \label{sec2}

Let $\mathbf{Y}_{n}=(Y_{1},\ldots,Y_{n})$ be a collection of $n$ random samples under the Bernoulli product model with unknown feature probabilities $(p_{j})_{j\geq1}$ such that $\sum_{j\geq1}p_{j}<+\infty$. That is, $Y_{i}=(Y_{i,j})_{j\geq1}$ is a sequence of independent Bernoulli random variables, with $p_{j}$ being the success probability of $Y_{i,j}$ for any $i=1,\ldots,n$, and $Y_{r}$ is independent of $Y_{s}$ for any $r\neq s$. Note that the assumption $\sum_{j\geq1}p_{j}<+\infty$ implies that each $Y_{i}$ displays finitely many features almost surely; indeed, by monotone convergence, we have $\sum_{j\geq1}p_{j}<+\infty$ if and only if $\mathbb{E}(\sum_{j\geq1}Y_{i,j})<+\infty$, which in turns implies that $\sum_{j\geq1}Y_{i,j}<+\infty$ almost surely. If $X_{n,j}$ denotes the number of times that feature $F_{j}$ has been observed in the random sample $\mathbf{Y}_{n}$, then
\begin{displaymath}
K_{n,r}=\sum_{j\geq1}\mathbb{I}{\{ X_{n, j} =r\}}
\end{displaymath}
is the number of features with frequency $r$ in $\mathbf{Y}_{n}$. Let $K_{n}$ be the total number of features in $\mathbf{Y}_{n}$, i.e. $K_{n}=\sum_{r\geq1}K_{n,r}$. For any two sequences $(a_n)_{n \geq 1}$ and $(b_n)_{n \geq 1}$, write $a_n\simeq b_n$ to mean that $a_{n}/b_{n}\rightarrow1$ as $n\rightarrow+\infty$. An intuitive estimator of $M_{n}$ can be deduced from a comparison between expectations of $M_{n}$ and $K_{n,1}$. Specifically, since $X_{n,j}$ is a Binomial random variable with parameter $(n,p_{j})$,
\begin{align*}
\mathbb{E}(M_{n}) & = \sum_{j \geq 1} p_j \mathbb{P} (X_{n,j} = 0) = \sum_{j \geq 1} p_j(1-p_j)^{n} \\
& =\frac{1}{{n+1 \choose 1}} \sum_{j \geq 1} {n+1 \choose 1} p_j(1-p_j)^{(n+1)-1} = \frac{1}{n+1} \mathbb{E} (K_{n+1,1}) \simeq \frac{1}{n} \mathbb{E} (K_{n,1})
\end{align*}
and set
\begin{displaymath}
\hat{M}_{n}=\frac{K_{n,1}}{n}
\end{displaymath}
as an estimator of $M_{n}$, for any $n\geq1$. The estimator $\hat{M}_{n}$ is nonparametric, in the sense that it does not rely on any distributional assumptions on the feature probabilities $(p_{j})_{j\geq1}$. See \citet{Cha(17)} for a somehow related derivation of $\hat{M}_{n}$.

The estimator $\hat{M}_{n}$ turns out to have the same analytic form of the popular Good-Turing estimator of the missing mass for species sampling models. The Good-Turing estimator first appeared in \citet{Goo(53)} as a nonparametric empirical Bayes estimator under the classical multinomial model for species sampling, i.e., $(Y_{1},\ldots,Y_{n})$ are $n$ random samples from a population of individuals belonging to a (possibly infinite) collection of species $(S_{j})_{j\geq1}$ with unknown proportions $(p_{j})_{j\geq1}$ such that $\sum_{j\geq1}p_{j}=1$. Under the multinomial model every observation is endowed with one species selected from $(S_{j})_{j\geq1}$, and hence $K_{n,1}\in\{0,1,\ldots,n\}$. Therefore, although the estimator $\hat{M}_{n}$ has the same analytic form of the Good-Turing estimator, the two estimators have different ranges (supports): while the Good-Turing estimator for species sampling models takes values in $[0,1]$, the Good-Turing estimator for feature allocation models takes positive (finite) values.

Hereafter we show that: i) the estimator $\hat{M}_{n}$ admits natural interpretations both as a jackknife (resampling) estimator and as a nonparametric empirical Bayes estimator in the sense of \citet{Efr(73)}; ii) the mean squared error of $\hat{M}_{n}$ converges to zero at the best possible rate for the estimation of the missing mass $M_{n}$; iii) the estimator $\hat{M}_{n}$ is linked to the Good-Turing estimator for species sampling models in terms of the limiting Poisson regime of the Binomial distribution.

\subsection{Interpretations of $\hat{M}_{n}$} \label{interprMn}
We first show that the estimator $\hat{M}_{n}$ admits a natural interpretation as a jackknife estimator in the sense of \citet{Que(56)} and \citet{Tuk(58)}. See also the monograph by \citet{Efr(87)} and references therein for a comprehensive account on jackknife (resampling) estimators. Let  $K_{n}(j)$ denote the number of distinct features observed in the sample $\mathbf{Y}_{n}$ after the removal of the $j$-th sample $Y_j$. It is easy to show that the missing mass $M_{n-1}$ equals the posterior expected value of $K_{n}-K_{n-1}(j)$, given all the samples but $Y_j$, i.e., $M_{n-1}=\mathbb{E}[K_{n}-K_{n-1}(j)| \left\{ Y_1, \ldots , Y_n \right\} \setminus Y_j]$. Accordingly, the difference $K_{n}-K_{n-1}(j)$ provides with an unbiased estimator of $M_{n-1}$. The jackknife estimator is then $\sum_{1\leq j\leq n}(K_{n}-K_{n-1}(j))/n$ which equals $\hat{M}_{n-1}=K_{n,1}/n$; indeed, for any fixed $j$, $K_{n}-K_{n-1}(j)$ coincides with the number of features displayed only by the $j$-th sample $Y_j$.

The estimator $\hat{M}_{n}$ also admits a natural interpretation as a nonparametric empirical Bayes estimator in the sense of \citet{Efr(73)}. A Bayesian nonparametric approach to estimate $M_{n}$ relies on the specification of a prior distribution for the unknown feature probabilities $(p_{j})_{j\geq1}$. Here we consider the three parameters Beta process prior (e.g., \citet{Teh(09)} and \citet{Jam(17)}). This is a well-known generalization of the celebrated Beta process prior of \citet{Hjo(90)}, and it is defined as the distribution of a completely random measure (\citet{Dal(08)}) with intensity measure $\nu(\D s,\D f)=\rho(s)\D s\mathbb{I}{(0,1)}(f)\D f$ where $\rho(s)=\vartheta s^{-1-\alpha}(1-s)^{\beta+\alpha-1}/\Gamma(1-\alpha)$, for $\alpha\in(0,1)$, $\beta>-\alpha$ and $\vartheta>0$, with $\Gamma(\cdot)$ being the Gamma function. In particular. under the three parameters Beta process prior, the Bayesian point estimator of $M_{n}$ under squared loss function is
\begin{equation}\label{eq:bnp}
\tilde{M}_{n}(\vartheta,\alpha,\beta)=\frac{\vartheta \Gamma(\alpha+\beta+n)}{\Gamma (\beta+n+1)}.
\end{equation}
We refer to the Appendix \ref{appEB} for further
details on the three parameters Beta process prior, and on the derivation of \eqref{eq:bnp}. In particular, it is shown in Appendix \ref{appEB} that 
\begin{displaymath}
\hat{\vartheta}_n=\frac{K_{n,1}\Gamma(\beta+ n+1)}{\Gamma(\beta + \alpha+n) n}
\end{displaymath}
is a consistent estimator of $\vartheta$ for any $\alpha \in (0,1)$ and $\beta>-\alpha$, and plugging $\hat{\vartheta}_n$ into \eqref{eq:bnp} we obtain
\begin{displaymath}
\tilde{M}_{n}(\hat{\vartheta}_{n},\alpha,\beta)  = \frac{K_{n,1}}{n} = \hat{M}_n.
\end{displaymath}
In other terms, the proposed estimator $\hat{M}_n$ coincides with the nonparametric empirical Bayes estimator $\tilde{M}_{n}(\hat{\vartheta}_{n},\alpha,\beta)$ of the missing mass. $\tilde{M}_{n}(\hat{\vartheta}_{n},\alpha,\beta)$ is obtained by assigning the three parameters Beta process prior to $(p_{j})_{j\geq 1}$ and then estimating its mass parameter $\vartheta$ with an appropriately chosen consistent estimator.

\subsection{Optimality of $\hat{M}_{n}$} \label{Optimality}

We start by showing that the mean squared error ($L_2$-risk) of the estimator $\hat{M}_{n}$ goes to zero at the best possible rate. This result legitimates the use of $\hat{M}_{n}$ as an estimator of the missing mass $M_{n}$. For any $0<W<+\infty$ let $\Pcr_W = \{  (p_j)_{j\geq1} : p_{j}\in(0,1)\text{ and }\, \sum_{j \geq 1}p_j \leq W\}$ be a class of features probabilities. Also, let $\hat{T}_n$ denote a whichever estimator of the missing mass $M_n$ based on the random sample $\mathbf{Y}_{n}$, that is a map on the $n$-fold product space of observations $\{0,1\}^{\infty}\times\cdots\times\{0,1\}^{\infty}$ and taking values in $\R^+$. For a specific collection of feature probabilities $(p_{j})_{j\geq1}$ in the class $\Pcr_W$, the $L_2$-risk of the estimator $\hat{T}_n$ is defined as follows
\begin{equation*}
R_{n} (\hat{T}_{n}; (p_j)_{j\geq1}) = \mathbb{E} [(M_{n}-\hat{T}_n)^2]=\Bias(\hat{T}_n)^2+\Var (\hat{T}_n-M_{n}), 
\end{equation*}
and the minimax risk is 
\begin{equation} \label{eq:minimax_def}
R^{\ast}_n = \inf_{\hat{T}_n}\sup_{(p_j)_{j\geq1} \in \Pcr_W} R_{n} (\hat{T}_{n}; (p_j)_{j\geq1}),
\end{equation}
i.e. the infimum, with respect to the set of possible estimators $\hat{T}_{n}$ of the missing mass $M_{n}$, of the worst-case risk over the class of feature probabilities $\Pcr_W$. The next theorem provides with an upper bound for $R_{n} (\hat{M}_n; (p_j)_{j\geq1}) $ and a lower bound for $R^{\ast}_n$.

\begin{thm} \label{thm1}
Let $(p_{j})_{j\geq1}$ be feature probabilities in the class $\Pcr_W$. Then,
\begin{itemize}
\item[i)] for any $n\geq1$
\begin{equation}\label{eq:MSE_p}
\sup_{(p_j)_{j\geq1} \in \Pcr_W} R_{n} (\hat{M}_{n}; (p_j)_{j\geq1}) \leq \frac{1}{n^2}W^2 + \frac{2n+1}{n(n+1)}W;
\end{equation}
\item[ii)]
for any $n \geq 2$
\begin{equation}\label{eq:lower}
R^{\ast}_{n} \geq  \frac{2W}{9(3n+1)} - \frac{14}{n^2} .
\end{equation}
\end{itemize}
\end{thm}

See Appendix \ref{AppProofThm1} for the proof of Theorem \ref{thm1}. For any two sequences $(a_n)_{n \geq 1}$ and $(b_n)_{n \geq 1}$, write $a_n\lesssim  b_n$ to mean that there exists a constant $C>0$ such that $a_n \leq C b_n$ for all $n$ and $a_n \asymp b_n$ to mean that we have both $a_n\lesssim  b_n$ and $b_n\lesssim  a_n$. Part i) of Theorem \ref{thm1} shows that $\sup_{(p_j)_{j\geq1} \in \Pcr_W} R_{n} (\hat{M}_{n}; (p_j)_{j\geq1}) \lesssim W/n$, namely as $n$ goes to infinity the mean squared error of $\hat{M}_{n}$ goes to zero at rate $W/n$. Part ii) of Theorem \ref{thm1} provides with a lower bound for the minimax risk, showing that $R^{\ast}\gtrsim W/n$. In other terms, this shows that the mean squared error of $\hat{M}_{n}$ goes to zero at the best possible rate. Theorem \ref{thm1} then shows that the estimator $\hat{M}_{n}$ is essentially rate optimal due to the matching minimax lower bound in the class $\Pcr_W $ of admissible probabilities' masses up to a constant factor.

In Theorem \ref{thm1} we considered the class $\Pcr_W$ of feature probabilities having total mass bounded by $W$. We highlight the crucial role or this assumption in order to provide asymptotic results. First, notice  from Theorem \ref{thm1} that, for a fixed value of the sample size $n$, the minimax rate increases linearly in $W$. This implies that, for a fixed $n$, the estimation problem can be made as hard as desired if no bounds are imposed on the sum of the possible vectors of probabilities $(p_j)_{j\geq 1}$. Since at first glance this result can seem slightly counter intuitive, let us consider an illustrative example which should clarify why the estimation difficulty of the problem is proportional to $W$. To simplify the argument, suppose that $W$ is a strictly positive integer and consider $W$ frequency vectors $(p^{(0)}_j)_{j\geq 1},\ldots,(p^{(W-1)}_j)_{j\geq 1} \in \Pcr_1$, i.e. such that $\sum_j p^{(r)}_j \leq 1$ for all $r\in \{0,\ldots,W-1 \}$. Now, let us construct another vector of frequencies $P = (p_j)_{j\geq 1} \in \Pcr_W$ obtained by setting $p_{qW+r} = p^{(r)}_q$ for $q \geq 0$ and $ 0 \leq r \leq W-1$. The resulting vector $P$ is the concatenation of $(p^{(0)}_j)_{j\geq 1},\ldots,(p^{(W-1)}_j)_{j\geq 1}$, such that in the first $W$ entries of $P$ we put the first elements $(p^{(0)}_1,\ldots,p^{(W-1)}_1)$, followed by the second ones 
$(p^{(0)}_2,\ldots,p^{(W-1)}_2)$, the third ones $(p^{(0)}_3,\ldots,p^{(W-1)}_3)$ and so on. Furthermore, any observation $(Y_{i,j})_{j\geq 1}$ sampled from the Bernoulli product model with frequencies given by this $P$ can be rewritten, following a similar construction, as the concatenation of $W$ observations $Y^{(r)}_{i}=(Y^{(r)}_{i,j})_{j\geq 1}$, $r\in \{0,\ldots,W-1\}$, each of them sampled from a binomial product model with the corresponding frequencies $(p^{(r)}_j)_{j\geq 1}$. Hence, the missing mass for a random sample from  $P$ can be related to the missing masses of each of its $W$ subcomponents, by  
 \begin{equation*}
M_n(\mathbf{Y}_{n};P) = \sum\limits_{r=0}^{W-1} M^{(r)}_n(\mathbf{Y}^{(r)}_{n};(p^{(r)}_j)_{j\geq 1}),
\end{equation*}
where $\mathbf{Y}^{(r)}=(Y_{1}^{(r)},\ldots,Y_{n}^{(r)})$. From this construction, we can see that by trying to estimate the missing mass on the left hand side, we are basically trying to estimate a sum of $W$ unrelated quantities, which explains why the error is linear in $W$.

In order to apply Theorem \ref{thm1} and evaluate the performances of the Good-Turing estimator compared to the minimax rate, we would need to know an upper bound $W$, i.e. an upper bound on the total mass of the unknown vector of probabilities $(p_j)_{j\geq 1}$. In real life applications, $W$ is unlikely to be known. However, we can easily estimate it. Specifically, since the total mass $W$ is the expected number of features displayed per observation,  we can use as a consistent estimator for $W$ the estimator $\hat{W}_n:=\sum_{j\geq 1} X_{n,j}/n$. Besides the following concentration inequality for $\hat{W}_n$ around its mean $W$ may be useful to measure the probability that 
$W_n$ deviates from $W$. In particular it can be proved that for any $\delta \in (0,1)$
\begin{equation}
\label{eq:W_concentration_right}
\mathbb{P} \left( \hat{W}_n -W > \sqrt{\frac{4W}{n} \log (1/\delta)} +\frac{1}{n}\log (1/\delta)\right) \leq \delta  
\end{equation}
and
\begin{equation}
\label{eq:W_concentration_left}
\mathbb{P} \left(- ( \hat{W}_n -W) > \sqrt{\frac{2W}{n} \log (1/\delta)}\right) \leq \delta .
\end{equation}
The proofs of these exponential tail bounds rely on suitable bounds on the corresponding log-Laplace transform of $\hat{W}_n -\mathbb{E} \hat{W}_n$. See Section \ref{AppProofConcW} for details. Equation \eqref{eq:W_concentration_left} leads to an upper bound on $W$, indeed with probability $1-\delta$ one has
\begin{equation}\label{eq:W_upper_bound}
W \leq \left( \sqrt{\hat{W}_n + \frac{\log ( 1/\delta )}{2n}} + \sqrt{\frac{\log ( 1/\delta )}{2n}} \right)^2.
\end{equation}
Analogously, one may directly apply  Equation \eqref{eq:W_concentration_right} in order to find a lower bound for $W$. More precisely, it is straightforward to see the validity of the following inequality 
\begin{equation} \label{eq:W_lower_bound}
\sqrt{W} \geq \sqrt{\hat{W}_n} - \sqrt{\frac{\log (1/\delta)}{n}}
\end{equation}
with probability bigger than $1-\delta$. Therefore \eqref{eq:W_upper_bound}--\eqref{eq:W_lower_bound} result in a confidence intervals for $W$.

Finally, we conclude this subsection by remarking that the proposed estimator displays some bias. The formula of the theoretical bias can be found in the proof of Theorem \ref{thm1}. However, the effect of this bias becomes negligible for any reasonable sample size. As a simple illustration, in Table \ref{Table.Bias}, we display the theoretical bias of the estimator for different sample sizes, when the frequencies $(p_{j})_{j\geq 1}$ are chosen according to the Zipf's Law for different values of the parameter $s$ (the same simulation setting will be reconsidered and better described in Section \ref{NumIll}). In Table \ref{Table.Risk}, we report for each combination of sample size $n$ and parameter $s$ the percentage of the risk that is due to the bias. As can be seen from Table \ref{Table.Bias} and \ref{Table.Risk}, expect for very small sample size and Zipf parameter, under which most of the frequencies are high, the bias is always almost negligible and does not contribute much to the risk of estimator, which is mainly due to its variance.

\begin{center}
\begin{table}[h!]
\small
\centering
\caption{Bias of $\hat{M}_n$: $p_j = 1/j^{s}$ for $j \leq 10^5$ and $p_j =0$ for $j>10^5$, and different values of $s$. }
	
\begin{tabular}{|c|cccc|} 
\hline  $s$ &  $n=10$ & $n=50$ & $n=100$ & $n=1000$ \\ \hline
0.6  & 1.310 & 0.459 & 0.241 & 0.001 \\
0.8 & 0.268 & 0.075 & 0.042 & 0.003 \\
1.0 & 0.100 & 0.020 & 0.010 & 0.001 \\
1.2 & 0.052 & 0.008 & 0.004 & 0.000 \\
1.4 & 0.033 & 0.004 & 0.002 & 0.000 \\
1.6 & 0.023 & 0.002 & 0.001 & 0.000 \\ \hline
\end{tabular}
\label{Table.Bias}		
\end{table}
\end{center}

\begin{center}
\begin{table}[h!]
\small
\centering
\caption{$Bias(\hat{M}_n)^2/R_{n}$*100: $p_j = 1/j^{s}$ for $j \leq 10^5$ and $p_j =0$ for $j>10^5$, and different values of $s$. }
	
\begin{tabular}{|c|cccc|} 
\hline  $s$ &  $n=10$ & $n=50$ & $n=100$ & $n=1000$ \\ \hline
0.6  & 16.06 & 13.05 & 9.57 & 0.89 \\
0.8 &  2.99  & 1.44 & 1.08 & 0.19 \\
1.0 & 1.35 & 0.34 &  0.19 & 0.03 \\
1.2 & 0.94 & 0.16 & 0.07 & 0.01 \\
1.4 & 0.76 & 0.10 &  0.04 & 0.00 \\
1.6 & 0.64 & 0.07 & 0.03 &  0.00 \\ \hline
\end{tabular}
\label{Table.Risk}		
\end{table}
\end{center}


\subsection{Connection to the Good Turing estimator for species sampling models} \label{ConnectionGT}

We relate the Good-Turing estimator for feature allocation models with the classical Good-Turing estimator for species sampling. This link relies on the well-known limiting Poisson approximation of Binomial random variables. Theorem \ref{thm1} states that, in the feature allocation models, the Good-Turing estimator achieves a risk of order $R_n \asymp \frac{W}{n}$, while it is known from \citet{Raj17} that the risk $\tilde{R}_n$ of the Good Turing estimator in the species sampling case asymptotically behaves as $\tilde{R}_n \asymp 1/n$. In order to compare the two models, we will consider the limiting scenario when $W \rightarrow 0$. Let us consider a vector of feature frequencies $(p_j)_{j\geq 1}$, such that $\sum_j p_j = W = \frac{\lambda}{n}$ for some positive value $\lambda$ and denote $\tilde{p}_j = \frac{p_j}{\sum_{j\geq 1} p_j}$ the normalized probability vector. Applying the large $n$ Poisson approximation of the binomial distribution, it follows that each $X_{n,j}$ is now approximately distributed according to a Poisson distribution with mean $\lambda \tilde{p}_j$. Therefore, the approximated model for large $n$ boils down to sample first an effective size $\tilde{n} = \sum_{j\geq 1} X_{n,j}$, distributed according to a Poisson with parameter $\lambda$, and, conditionally on it, sample $\tilde{n}$ independent and identically distributed observations $(\tilde{Y}_1,\ldots,\tilde{Y}_{\tilde{n}})$ from probability vector $(\tilde{p}_j)_{j\geq 1}$. Hence it is equivalent to a species sampling model where the sample size $\tilde{n}$ is assumed to be random. Denote $\tilde{M}_{\tilde{n}}$ the missing mass in the corresponding species sampling model. Now, by noticing that $M_n = \frac{\lambda}{n} \tilde{M}_{\tilde{n}}$, it follows that $R_n \approx \frac{\lambda^2}{n^2} \mathbb{E}_{\tilde{n}} (\tilde{R}_{\tilde{n}})$. Therefore, we expect the quantity $\frac{n^2}{\lambda^2} R_n $ to have the same asymptotic behaviour of $\mathbb{E}_{\tilde{n}} \left(\frac{1}{\tilde{n}} \mathbb{I}{\{ \tilde{n} > 0\}}\right)$. Indeed, from Theorem \ref{thm1} and $W = \frac{\lambda}{n}$, we get that $\frac{n^2}{\lambda^2} R_n \asymp \frac{1}{\lambda} = \frac{1}{\mathbb{E} (\tilde{n})}$. 

We conclude by remarking that the above construction also provides a justification for the Good-Turing estimator in the feature allocation models. Indeed, in the context of feature allocation models, we want to estimate $M_n = W\  \tilde{M}_{\tilde{n}}$, where both $W$ and $\tilde{M}_{\tilde{n}}$ are unknown and need to be estimated. In order to estimate $\tilde{M}_{\tilde{n}}$, we can use the Good-Turing estimator for species models, which here turns to be 
\begin{equation*}
\tilde{M}_{\tilde{n}}^{\text{GT}} = \frac{K_{n,1}}{\tilde{n}}= \frac{K_{n,1}}{\sum_{j\geq 1} X_{n,j}}.
\end{equation*}
However, we also need to estimate $W = \sum_j p_j$ in order to make use of $\tilde{M}^{GT}_{\tilde{n}}$ as an estimator of $M_n$. As pointed out at the end of the previous subsection, a consistent estimator of $W$ is  $\hat{W} = \frac{\sum_j X_{n,j}}{n}$. Then, by combining the estimator $\hat{W}$ for $W$ with the estimator $\tilde{M}_{\tilde{n}}^{\text{GT}}$ for $\tilde{M}_{\tilde{n}}$, we obtain the following estimato of $M_n$,
\begin{equation*}
\hat{M}_{n}= \hat{W} \tilde{M}^{\text{GT}}_{\tilde{n}} = \frac{\sum_{j\geq 1} X_{n,j}}{n} \frac{K_{n,1}}{\sum_{j\geq 1} X_{n,j}} = \frac{K_{n,1}}{n},
\end{equation*}
which turns out to be exactly the Good-Turing estimator for feature allocation models. \\

\section{A confidence interval for $\hat{M}_{n}$} \label{ConfidenceInt}

In this section, we consider the problem of uncertainty quantification of the proposed estimator of the missing mass. In particular, we exploit tools from concentration inequalities for sum of independent random variables (\citet{Bou(13)}) to introduce a non-asymptotic level-$\delta$ confidence interval for $\hat{M}_{n}$, for any $\delta\in(0,1)$.

\bigskip

\begin{thm} \label{thm2}
Let $(p_{j})_{j\geq1}$ be any sequence of feature probabilities s.t. $\sum_{j\geq1}p_{j}<+\infty$, and set $c_{\delta}(x)=(  \sqrt{\log(1/\delta)/2} +\sqrt{7\log(1/\delta)/6 +x})^2$ for any nonnegative integer $x$ and $\delta\in(0,1)$. Then, with probability at least $1-\delta$
\begin{equation*}
\hat{M}_n-L_{\delta}(K_{n,1},K_{n,2})\leq M_{n}\leq \hat{M}_{n}+U_{\delta}(K_{n}),
\end{equation*}
where
\begin{equation*}
L_{\delta}(K_{n,1},K_{n,2})=\frac{2c_{\delta}(K_{n,2})}{n(n-1)}+\frac{\log(6/\delta)}{n}+\sqrt{2\log\left(\frac{6}{\delta}\right)\left(\frac{4c_{\delta}(K_{n,1})}{n(n-1)}+\frac{2c_{\delta}(K_{n,2})}{n^{2}}\right)}
\end{equation*}
and
\begin{equation*}
U_{\delta}(K_{n})=\frac{\log(6/\delta)}{n-1}+\sqrt{2\log\left(\frac{6}{\delta}\right)\frac{4c_{\delta}(K_{n})}{(n-1)^{2}(1-2/n)}}.
\end{equation*}
\end{thm}

\bigskip
\bigskip

This theorem can be proven in two steps. We first derive bounds on the log-Laplace transform of $D_n = \hat{M}_n -M_n$, from which we can obtain a confidence interval for $D_n$ as a function of $\mathbb{E} K_{n,1}$, $\mathbb{E} K_{n,2}$ and $\mathbb{E} K_n$. Then, by deriving bounds on the log-Laplace of $K_{n,1}$, $K_{n,2}$ and $K_n$, we are able to bound with high probability the deviation of these random variables from their mean. Theorem \ref{thm2} then follows by plugging-in the results of step 2 in the bounds of step 1. The details are left in Appendix \ref{AppProofThm2} . For any $\delta\in(0,1)$, Theorem \ref{thm2} provides with a level $1-\delta$ confidence interval for $\hat{M}_{n}$ that holds for every sample size $n$ and for every possible values of the feature probabilities $(p_{j})_{j\geq1}$. Indeed, this is derived by applying finite sample concentration inequalities, without using any asymptotic approximation, and it does not rely on any distributional assumption on the feature probabilities $(p_{j})_{j\geq1}$. Note that the proposed confidence interval can be easily evaluated without resorting to any simulation based strategy. It is enough to count the total number of features and number of features with frequency one and two in the observable sample, i.e. $K_{n}$, $K_{n,1}$ and $K_{n,2}$, and plug these  quantities into $U_{\delta}(K_{n})$ and $L_{\delta}(K_{n,1},K_{n,2})$ to be added and subtracted to $\hat{M}_{n}$.

\section{A Stopping Rule for the Discovery Process} \label{StoppRule}

As we recalled in the Introduction, interest in estimating missing mass $M_{n}$ is motivated by the design of feature inventories that are cost-effective in terms of the number of features discovered and the amount of resources allocated in the discovery process. This is a fundamental issue in numerous scientific disciplines where the sampling procedure is expensive, in terms of time and/or financial resources. Feature inventories must then be designed in such a way to redirect the search of new features to more productive sites, methods or time periods whenever the sampling effort becomes unprofitable. In such a context, the estimator $\hat{M}_{n}$ is the key ingredient for defining an adaptive sequential approach in terms a stopping rule for the discovery process. Specifically, let $h$ be an utility function, defined on the integers, such that $h(k)$ is the gain of observing $k$ features; assume that $h$ in non-decreasing and concave. If $c$ is the cost associated with each sampling step, and $K_{n}$ is the number of features in the sample $\mathbf{Y}_{n}$, then the stopping rule may be defined as
\begin{equation*}\label{stop}
n^{\ast}=\inf\{n\geq0\,:\,[h(K_{n}+\hat{M}_{n})-h(K_{n})]\leq c\}.
\end{equation*}
This brief discussion highlights how to exploit the estimator $\hat{M}_{n}$ within the context of designing cost-effective feature inventories. In particular, it gives rises to the challenging problem of finding the time $n^{\ast}$ at which it is optimal to conclude the discovery process, i.e. the time that maximizes the expected payoff $\mathbb{E}(h(K_{n})-cn)$.

\section{Numerical Illustration} \label{NumIll}

In this section we illustrate the experimental performance of the estimator $\hat{M}_{n}$ by the analysis of various synthetic data, and by the analysis SNP datasets from a genome project. In particular, let $N$ denote the total number of possible features in the whole population of individuals. With regards to synthetic data, we present a numerical illustration by setting $p_j = 1/j^{s}$ for $j \leq N$ and $p_j =0$ for $j>N$, with $s$ being a nonnegative parameter.  Note that the feature probability masses $p_{j}$'s correspond to the unnormalized masses of the ubiquitous Zipf distribution. Recall that the parameter $s$ controls how the total mass is spread among the features: the lager $s$, the smaller is the number of features with high probability. In other terms the larger $s$, the larger is the number of features with small frequencies (rare features). Hereafter we set $N=10^5$, and we consider the following values for the discount parameter $s$: 0.6, 0.8, 1.0, 1.2, 1.4, 1.6. For each of these values of the parameter $s$, we take $100$ samples of size $n=50, 250, 1000$ from the population $(p_{j})_{1\leq j\leq N}$. Table 1 displays the true value of the missing  mass $M_{n}$, the estimated value $\hat{M}_{n}$ and the corresponding $95\%$ confidence interval (CI). All the experiments are averaged over the 100 samples. Results in Table 1 show that $\hat{M}_n$ provides good estimates of true value of $M_n$ in all the scenarios that we considered. In addition, confidence intervals are quite narrow around the estimator and contain always the true value of the missing mass.

%

\begin{center}
\begin{table}
\small
\centering
\caption{Synthetic data: missing mass estimation with $p_j = 1/j^{s}$ for $j \leq 10^5$ and $p_j =0$ for $j>10^5$, and different values of $s$. }
\begin{tabular}{|c|ccc|ccc|ccc|} \hline\label{tab:zipf}
$s$ & \multicolumn{3}{c|}{$n=50$} & \multicolumn{3}{c|}{$n=250$}  \\ \hline\hline
			 & $M_{n}$ & $\hat{M}_n$ & $95\%$ CI  & $M_{n}$ & $\hat{M}_n$ & $95\%$ CI
			\\ \hline
			0.6  & 183.81 & 184.66 & (174.35,\,198.12)   & 105.38 & 105.61 & (101.74,\,110.53) \\
			0.8 & 33.79 & 33.67 & (29.45,\,39.51) & 26.40 & 26.05 & (24.40,\,28.30) \\
			1.0 & 7.02 & 7.02 & (4.88,\,9.91) & 5.41  &  5.43 & (4.66,\,6.49)\\
			1.2 & 1.92 & 1.87 & (0.54,\,3.60)  & 1.35 & 1.35 & (0.92,\,1.93) \\
			1.4 & 0.71 & 0.72 & (0,\,1.98) & 0.44 & 0.44 & (0.15,\,0.81) \\
			1.6 & 0.34 & 0.36 & (0,\,1.39) & 0.18 & 0.18 & (0,\,0.46)\\ \hline
			
\end{tabular}	
\begin{tabular}{|c|ccc|} \hline
$s$ & \multicolumn{3}{c|}{$n=1000$} \\ \hline\hline & $M_{n}$ & $\hat{M}_n$ & $95\%$ CI \\ \hline
0.6  & 27.84 & 27.89 & (26.66,\,29.65) \\
0.8 & 17.18 & 17.17 & (16.47,\,18.13) \\
1.0 & 4.04 & 4.03 & (3.70,\,4.48) \\
1.2 & 0.97 & 0.97 & (0.80,\,1.20) \\
1.4 & 0.29 & 0.29 & (0.19,\,0.43) \\
1.6 & 0.12 & 0.11 & (0.04,\,0.21) \\ \hline
\end{tabular}
			
\end{table}
\end{center}

We conclude with an application to SNP  data. In this context, each sample represents a genome sequence and we are interested in variations with respect to the reference genome.  Single nucleotide polymorphisms (SNPs) are the most common type of genetic variation among people. Each SNP represents a difference in a single DNA building block, which is called a nucleotide.
For example, at a specific position in the human genome, the C (Cytosine)  nucleotide may appear in most individuals, but the position can be replaced by an A (Adenine) in a small group of people. This means that there is a SNP at that position of the DNA. SNPs are important biological markers, helping scientists locate genes that are associated with diseases.
We use the SNP datasets from the ENCODE sequencing genome project (http://www.hapmap. org/downloads/encode1.html.en). The same project was analyzed in \citet{Ion(09)}.
Ten 500-kb regions of the genome were sequenced in 209 unrelated DNA samples: 60 Yoruba (YRI), 60 CEPH European (CEPH), 45 Han Chinese (CHB), and 44 Japanese (JPT). These regions were chosen to be representative of the genome in general, including various chromosomes, recombination rates, gene density, and values of nontranscribed conservation with mouse.
To make results comparable across the 4 populations (YRI, CEPH, CHB, and JPT), we consider only $n=20$ of the sequenced individuals for each dataset. Table 2 displays the estimated values and $95\%$ confidence interval (CI). For samples of 20 individuals, the YRI population displays the highest estimate of the missing mass. This agrees with results in \citet{Ion(09)}, showing that the African population is the most diverse. We also consider increasing the sample size to $n=40$ of the sequenced individuals for each dataset. Table 2 shows how the missing mass decreases with respect to the case $n=20$. This highlights the saturation effect in discovering new variants. The discovery process is very efficient in the beginning, but after many individuals are sequenced, each additional individual contributes less and less to the pool of the newly discovered variants.

\begin{center}

\begin{table}
\small
\centering
\caption{Real data: missing mass estimation for the ENCODE sequencing genome project.}
\begin{tabular}{|c|cc|cc|} \hline
\text{Population} & \multicolumn{2}{c|}{$n=20$} & \multicolumn{2}{c|}{$n=40$}  \\ \hline\hline
			 & $\hat{M}_n$ & $95\%$ CI  &  $\hat{M}_n$ & $95\%$ CI \\ \hline
			\text{CEPH} & 55.6 & (38.7,\,88.9) & 22.3 & (15.6,\,38.9) \\
			\text{CHB } & 50.0 & (37.3,\,81.3) & 17.6 & (12.6,\,32.8) \\
			\text{JPT} & 61.7 & (46.9,\,93.4) & 26.7 & (20.9,\,42.2)\\
			\text{YRI} & 88.3 & (65.2,\,125.2) & 26.6 & (18.9,\,44.7) \\ \hline
\end{tabular}
			
\end{table}

\end{center}

\section{Concluding Remarks} \label{Conclusion}
In this paper we delved into the Good-Turing estimator for feature allocation models, thus providing theoretical guarantees for its use. In particular, we proved that the mean squared error of $\hat{M}_{n}$ goes to zero at the best possible rate as the sample size $n$ goes to infinity. Our results are simple, intuitive and easily implementable from practitioners. It distinguishes from the approaches of \citet{Ion(09)} and \citet{Zou(16)} for being the first nonparametric statistical approach for estimating the missing mass in feature allocation models. In particular, differently from \citet{Ion(09)} and \citet{Zou(16)}, we associated to the proposed estimator an exact (non-asymptotic) quantification of its uncertainty. Results and techniques presente in this paper paves the way to new research directions in the context of feature allocation problems. In particular, three promising directions are: i) estimating the conditional expected number, given an observable sample $\mathbf{Y}_{n}=(Y_{1},\ldots,Y_{n})$, of features with frequency $r>0$ in $\mathbf{Y}_{n}$ that would be observed in one additional sample (\citet{Cha(17)}); ii) solving the optimal stopping problem discussed in Section \ref{StoppRule}; iii) estimating the conditional expected number, given an observable sample $\mathbf{Y}_{n}$, of new features that would be observed in $m>1$ additional (unobservable) samples (\citet{Col(12)}).

\appendix

\section{Appendix - Proofs}\label{app}

\subsection{Nonparametric empirical Bayes} \label{appEB}

In the present section we derive $\hat{M}_{n}$ as a nonparametric empirical Bayes estimator in the sense of \citet{Efr(73)}. Recall that $F_j$ denotes the label of feature $j$, in the sequel assumed to be $F_j \in (0,1)$.  $Y_i=(Y_{i,j})_{j \geq 1}$ denotes the observation of $i$-th individual. Each entry $Y_{i,j}$ in this sequence is distributed according to a Bernoulli distribution  with parameter $p_j$ and is independent of the others. $Y_{i,j}=1$ (resp. $Y_{i,j}=0$) indicates the presence (resp. absence) of feature $j$ in the $i$--th individual.  The Bayesian nonparametric approach to estimate the missing mass $M_n$ requires a prior specification for the $p_j$'s: we resort to the three-parameter Beta process introduced by \cite{Teh(09)}. Such a prior distribution is defined as the distribution of a Completely Random Measure (CRM) $\tilde{\mu}$ (see e.g. \cite{Dal(08)}), with
L\'evy intensity  $\nu (\D s ,  \D f)= \rho (s)\D s \mathbb{I}{(0,1)}(f) \D f$, where $\rho (s) =
\vartheta/\Gamma (1-\alpha) s^{-1-\alpha} (1-s)^{\beta +\alpha-1}$, being $\alpha, \beta \in (0,1)$ and $\vartheta>0$.\\
Let $\mathbf{Y}_{n}$ be a random sample of size $n$ and $F_1^*, \ldots , F_{K_{n}}^{*}$ be the $K_{n}$ distinct features observed in it, the posterior estimate of $M_n$, under a square loss function, equals
\begin{equation}
\label{eq:posterior_estimator}
\tilde{M}_n (\vartheta , \alpha , \beta ) = E [\tilde{\mu} ((0,1)\setminus \left\{ F_1^*, \ldots , F_{K_{n}}^* \right\})| \mathbf{Y}_{n}].
\end{equation} 
We can characterize the posterior distribution of $\tilde{\mu}$ in \eqref{eq:posterior_estimator} resorting to 
 \cite[Theorem 3.1 (ii)]{Jam(17)}, which gives
\[
\tilde{\mu}| \mathbf{Y}_{n} \stackrel{d}{=} \tilde{\mu}_n +\sum_{\ell=1}^{K_{n}} J_\ell \delta_{F_\ell^*}
\]
where the $J_\ell$'s are jumps having a density on the positive real line and $\tilde{\mu}_n$ is a CRM having updated L\'evy intensity given by
$\nu_n (\D s \, \D f)=(1-s)^n\rho (s) \D s \mathbb{I}{(0,1)}(f ) $ $ \D f$. Hence the Bayesian estimator in \eqref{eq:posterior_estimator} boils down to
\begin{align*}
\tilde{M}_n (\vartheta , \alpha , \beta ) &= \mathbb{E} \left[ \tilde{\mu}_n (\left\{ F_1^*, \ldots , F_{K_{n}}^* \right\}^c) +
\sum_{\ell=1}^{K_{n}} J_\ell \delta_{F^*_\ell} (\left\{ F_1^*, \ldots , F_{K_{n}}^* \right\}^c) \right] \\
& = \mathbb{E} [\tilde{\mu}_n ((0,1))]
 = -\frac{\D }{\D u} \mathbb{E} [e^{-u \tilde{\mu}_n ((0,1))}]\Big|_{u=0} = \vartheta\frac{\Gamma (\alpha +\beta+n)}{\Gamma (\beta+n +1)}
\end{align*} 
where the second equality follows from the fact that the base measure of $\tilde{\mu}_{n}$ is diffuse and the last equality  follows from the availability of the Laplace functional and from standard calculations.  We are now
going to show that $\hat{\vartheta}_n = (K_{n,1}\Gamma(\beta+ n+1))/(\Gamma(\beta + \alpha+n) n)$ is a consistent estimator of $\vartheta$, then we will conclude that the empirical Bayes estimator  $\tilde{M}_n (\hat{\vartheta}_n , \alpha , \beta )$ coincides with $\hat{M}_n$. The consistency of $\hat{\vartheta}_n$ can be established resorting to the regular variation theory by 
\cite{Kar(67),Gne(07)}. More specifically we are able to show that
the tail integral of $\rho (s)$ is a regularly varying function, having regular variation index $\alpha$, indeed
\begin{equation*}
\overline{\rho}(s) := \int_x^{1} \rho (s) \D s= \frac{\vartheta}{\Gamma (1-\alpha)} \int_x^1 s^{-1-\alpha} (1-s)^{\beta+\alpha-1}\D s
\end{equation*}
fulfills
\begin{equation*}
\lim_{x \to 0} \frac{\overline{\rho}(s)}{x^{-\alpha}} =\frac{\vartheta}{\Gamma (1-\alpha)} \lim_{x \to 0} \frac{x^{-1-\alpha}
(1-x)^{\beta +\alpha -1}}{\alpha x^{-\alpha-1}} = \frac{\vartheta}{\Gamma (1-\alpha ) \alpha} .
\end{equation*}
Hence we just proved that $\overline{\rho} (x) \simeq x^{-\alpha} \vartheta/(\Gamma (1-\alpha ) \alpha)$ as $n \to +\infty$, and therefore an  application of \cite[Corollary 21]{Gne(07)} gives the asymptotic relation
$\frac{K_{n,1}}{n^\alpha} \rightarrow \vartheta  $. Moreover, from Stirling's approximation, $\frac{\Gamma(\alpha+\beta +n)}{\Gamma(n+1+\beta)}\simeq n^{\alpha-1}$. Therefore,  $\hat{\vartheta}_n = \frac{ K_{n,1}\Gamma(\beta+ n+1)}{\Gamma(\beta + \alpha+n) n} \rightarrow \vartheta$ and $\hat{\vartheta}_n$ is a consistent estimator for $\vartheta$. Plugging $\hat{\vartheta}_n$ into the posterior estimate $\tilde{M}_n (\vartheta , \alpha , \beta )$, we obtain $\tilde{M}_n (\hat{\vartheta}_n, \alpha , \beta )=\frac{K_{n,1}}{n}=\hat{M}_{n}$.\hfill \qed

\subsection{Proof of Theorem \ref{thm1}} \label{AppProofThm1}

We prove part i) (upper bound) and ii) (lower bound) of Theorem \ref{thm1} separately.\\
We first focus on  the upper bound for the $L_2$-risk of $\hat{M}_n$. 
In the sequel we denote by $M_{n,1}$ the total mass of features with frequency 1 in a sample of size $n$, formally
$
M_{n,1}= \sum_{j \geq 1} p_j\mathbb{I}{\{ X_{n,j} =1  \}},
$
we observe that the expectation of $M_{n,1}$ equals
\begin{equation*}
\mathbb{E} (M_{n,1}) = \sum_{j \geq 1} p_j \mathbb{P} (X_{n,j}=1)= \sum_{j \geq 1} {n \choose 1} p_j^{2} (1-p_j)^{n-1}
\end{equation*}
and obviously $\mathbb{E} [M_{n,1}] \leq W$, for any sequence $(p_j)_{j \geq 1}$ belonging to the class $\Pcr_W$.\\
In order to bound the $L_2$-risk of the estimator $\hat{M}_n$ from above for any $(p_j)_{j \geq 1}\in \Pcr_W$, we remind that $R_{n} (\hat{M}_{n}; (p_j)_{j\geq1}) =\mathbb{E} (\hat{M}_n-M_{n})^2$ may be seen as the sum of the squared bias and the variance. The upper bound for the bias can be easily proved as follows
\begin{align*}
\mathbb{E} (\hat{M}_n-M_{n})  & = \frac{1}{n} \sum_{j \geq 1} {n \choose 1} 
p_j^{1} (1-p_j)^{n-1} -\sum_{j \geq 1}  p_j (1-p_j)^{n} \\
& = \sum_{j \geq 1}  p_j (1-p_j)^{n-1}-\sum_{j \geq 1}p_j(1-p_j)^{n} =  \sum_{j \geq 1} p_j^{2} (1-p_j)^{n-1} \\
& = \frac{1}{ {n \choose 1}}\sum_{j \geq 1} {n \choose 1} p_j^2 (1-p_j)^{n-1} = \frac{1}{n}\mathbb{E}(M_{n,1}) \leq \frac{1}{n}W.
\end{align*}
As for  the variance bound,  we note that
\begin{align*}
\Var (\hat{M}_n-M_{n}) & = \sum_{j \geq 1} \Var \left( \frac{1}{n}
\mathbb{I}{\{ X_{n, j} = 1 \}} -p_j \mathbb{I}{\{ X_{n, j}=0 \}} \right) \\
&  \leq \sum_{j \geq 1} \mathbb{E} \left(  \frac{1}{n}
\mathbb{I}{\{ X_{n, j} = 1 \}} -p_j \mathbb{I}{\{ X_{n, j}=0 \}}\right)^2,
\end{align*}
where we have exploited the independence of the random variables $X_{n, j}$. We now observe that the events  $\{ X_{n, j}=1 \}$ and $\{ X_{n, j}=0 \}$ are incompatible, hence we get
\begin{align*}
\Var (\hat{M}_n-M_{n}) & \leq  \sum_{j \geq 1} \mathbb{E} \left(  \frac{1}{n}
\mathbb{I}{\{ X_{n, j} = 1 \}} -p_j \mathbb{I}{\{ X_{n, j}=0 \}}\right)^2\\
& =  \sum_{j \geq 1} \left( \left( \frac{1}{n} \right)^2 \mathbb{P} (X_{n, j}= 1)
+ p_j^2 \mathbb{P} (X_{n, j} = 0)\right)\\
& =  \left( \frac{1}{n} \right)^2 \sum_{j \geq 1} \mathbb{P} (X_{n, j}= 1)
+ \sum_{j \geq 1} p_j^2 \mathbb{P} (X_{n, j} = 0)\\
& =  \frac{1}{n^2}\sum_{j \geq 1} {n \choose 1} p_j (1-p_j)^{n-1}
+ \sum_{j \geq 1} p_j^2 (1-p_j)^n\\
& =  \frac{1}{n}\sum_{j \geq 1} p_j (1-p_j)^{n-1}
+ \frac{1}{{n+1 \choose 1}}\sum_{j \geq 1} {n+1 \choose 1} p_j^2 (1-p_j)^n\\
& = \frac{1}{n} \mathbb{E} (M_{n-1}) + \frac{1}{n+1} \mathbb{E} (M_{n+1, 1}) \leq \frac{2n+1}{n(n+1)}W.
\end{align*}
Putting together the bound on the bias and the variance we immediately obtain \eqref{eq:MSE_p}.\\

We now focus on the proof of the lower bound of the minimax risk $R^{\ast}_n$, which has been defined in \eqref{eq:minimax_def}. Let us first notice that, since $(p_j)_{j \geq 1}$ belongs to $\Pcr_W$, then 
$M_n \leq W $  almost surely, and hence $(M_n-\hat{T}_n)^2 \geq (M_n-\min (\hat{T}_n,W))^2$. Therefore, we can restrict the minimum in 
\eqref{eq:minimax_def} over all the possible estimators  less than $W$, hence the problem boils down to determine an estimate from below for  $\inf_{\hat{T}_n\leq W}\sup_{(p_j)_j \in \Pcr_{W}} \mathbb{E} [(M_{n}-\hat{T}_n)^2]$. 
The proof of this estimate relies on a Bayesian parametric approach  based on the randomization of  the probability masses $p_j$'s.
More specifically we consider  the randomized vector $\mathsf{P} = (\mathsf{p}_1, \cdots, \mathsf{p}_m, 0 , 0, \ldots)$ where $m$ is a Poisson random variable with mean $\lambda = nW$ and $\mathsf{p}_j$ are independent beta random variables with parameters $a=1$ and $b=2n-1$.
We denote by $\mathsf{P}_F := \sum_{j=1}^m \mathsf{p}_j$ the total mass of the random sequence $(\mathsf{p}_j)_{j \geq 1}$; we further define
$X:=(X_{n,j})_{j \geq 1}$ and it is worth noticing that, conditionally on $\mathsf{P}$, the $X_{n,j}$'s are independent, having a Binomial distribution with parameters $(n, \mathsf{p}_j)$. 
For generic random variables $U,V$ we will  write $E_V$ (resp. $E_{V|U}$) to mean that the expectation is made with respect to $V$ (resp. to $V$ given $U$).
It is easily shown that
\begin{align}
& \inf_{\hat{T}_n\leq W}\sup_{(p_j)_j \in \Pcr_{W}} \mathbb{E} [(M_{n}-\hat{T}_n)^2] \nonumber  \\
& \qquad\qquad \geq
\inf_{\hat{T}_n\leq W} \mathbb{E}_{\mathsf{P} } \left\{ \mathbb{I}{\{\mathsf{P} \in \Pcr_{W}\}} \mathbb{E}_{X | \mathsf{P}} [(M_{n}-\hat{T}_n)^2| \mathsf{P}]  \right\} \nonumber\\
& \qquad\qquad \geq \inf_{\hat{T}_n\leq W} \left\{ \mathbb{E}_{\mathsf{P}} \mathbb{E}_{X | \mathsf{P}} [(M_{n}-\hat{T}_n)^2| \mathsf{P}]  \right. \nonumber \\
 & \qquad\qquad\qquad\qquad \left. - \mathbb{E}_{\mathsf{P} } (1- \mathbb{I}{\{\mathsf{P} \in \Pcr_{W}\}}) \mathbb{E}_{X |\mathsf{P}} [(M_{n}-\hat{T}_n)^2| \mathsf{P}] \right\} \nonumber \\ 
& \qquad\qquad \geq \inf_{\hat{T}_n\leq W} \left\{  \mathbb{E}_{\mathsf{P} } \mathbb{E}_{X | \mathsf{P}} [(M_{n}-\hat{T}_n)^2| \mathsf{P}]
 - \mathbb{E}_{\mathsf{P} } (1- \mathbb{I}{\{\mathsf{P} \in \Pcr_{W}\}}) \mathsf{P}_F^2 \right\} \nonumber \\ 
& \qquad\qquad \geq \inf_{\hat{T}_n\leq W} \mathbb{E}_{\mathsf{P} } \mathbb{E}_{X |\mathsf{P}} [(M_{n}-\hat{T}_n)^2| \mathsf{P}] - \mathbb{E}(\mathsf{P}_F^2\mathbb{I}\{ \mathsf{P}_F > W\}) \label{eq:minmax2}
\end{align}
We bound separately the two terms on the r.h.s. of \eqref{eq:minmax2}. We start by deriving an upper bound for the term on the right. Using Fubini's Theorem, it comes that
\begin{align*}
& \mathbb{E}(\mathsf{P}_F^2 \mathbb{I}\{ \mathsf{P}_F > W\}) \\
& \qquad\qquad= \int s^2 \mathbb{I} \{s>W\} \mathbb{P}_{\mathsf{P}_F}( {\rm d} s) = 2 \int \int t \mathbb{I}\{ t < s\}  {\rm d} t \mathbb{I} \{s>W\} \mathbb{P}_{\mathsf{P}_F}( {\rm d} s) \\
& \qquad\qquad = 2 \int t \int  \mathbb{I}\{ t < s\} \mathbb{I} \{s>W\} \mathbb{P}_{\mathsf{P}_F}(ds)  {\rm d}  t 
= 2 \int t \mathbb{P}(\mathsf{P}_F > \max(t,W)) {\rm d} t
\end{align*}
which leads to the more convenient form
\begin{equation}
\mathbb{E}(\mathsf{P}_F^2 \mathbb{I}\{ \mathsf{P}_F > W\}) = W^2 \mathbb{P}(\mathsf{P}_F > W) +  2 \int_{W}^\infty t \mathbb{P}(\mathsf{P}_F > t)  {\rm d}  t. \label{P>F}
\end{equation}
In order to provide an upper bound for the probability of the event $\{ \mathsf{P}_F > t\}$, we use a Markov inequality with third centralized moment. We first evaluate the mean of the total mass, 
$$ \mathbb{E} \mathsf{P}_F  =  \mathbb{E} [\ \mathbb{E}(\sum\limits_{j=1}^m p_j | m)\ ] = \mathbb{E} \frac{ m a}{a+b} = \mathbb{E} \frac{ m }{2n} =  \frac{ W}{2}. $$
Now, denoting $\mu_3(Y)$ the centralized third moment of the r.v. $Y$, the law of total cumulant implies that
\begin{align*}
\mathbb{E} (\mathsf{P}_F-\mathbb{E} \mathsf{P}_F)^3 & = \mu_3(P_F) \\
& = \mathbb{E} [ \mu_3(\mathsf{P}_F | m) ] + \mu_3(\mathbb{E} [\mathsf{P}_F | m]) + 3 \text{Cov}(\mathbb{E} [\mathsf{P}_F | m], \Var [\mathsf{P}_F | m]) \\
& = \mathbb{E} [ \mu_3(\mathsf{P}_F | m) ] + \mu_3(\frac{ m }{2n}) + 3 \text{Cov} \Big(\frac{ m }{2n}, \frac{ m a b}{(a+b)^2(a+b+1)}\Big) \\
& = \mathbb{E} [ \mu_3(\mathsf{P}_F | m) ] + \frac{ 1 }{8n^3} \mu_3(m) + 3\frac{2n-1}{8n^3(2n+1)} \Var(m)
\end{align*}
$m$ is Poisson distributed, so its variance and third centralized moment are equal to its mean $\lambda=nW$. Further noticing that  $\frac{2n-1}{2n+1} < 1$, we upper bound the third centralized moment by $\mathbb{E} [ \mu_3(\mathsf{P}_F | m) ] + \frac{W}{2n^2}.$ Conditioning on $m$, the total mass $\mathsf{P}_F$ is a sum of independent and identically distributed random variables  so $\mu_3(\mathsf{P}_F | m) = m \mu_3(p_1)$. Since $p_1$ is Beta distributed with parameters $a$, and $b$, its third centralized moment is given by 
\begin{align*}
 \mu_3(p_1) &= [\Var (p_1)]^{3/2} \frac{2(b-a)\sqrt{a+b+1}}{(a+b+2)\sqrt{ab}} 
 				=  \frac{2ab(b-a)}{(a+b)^3(a+b+1)(a+b+2)} \\
			&= \frac{4 (2n-1) (n-1)}{(2n)^3(2n+1)2(n+1)} = \frac{(2n-1)(n-1)}{4n^3(2n+1)(n+1)}.
\end{align*}
Using the fact that $\frac{(2n-1)(n-1)}{(2n+1)(n+1)} \leq 1 \leq 2$, we successively find that $\mathbb{E} [ \mu_3(\mathsf{P}_F | m) ] \leq \frac{W}{2n^2}$ and finally $\mu_3(\mathsf{P}_F) \leq \frac{W}{n^2},$ obtaining the inequality we need to bound $\mathbb{P} (\mathsf{P}_F > t).$ Indeed, Markov's inequality implies that
\begin{align*}
\mathbb{P} (\mathsf{P}_F > t) &= \mathbb{P}\Big(\mathsf{P}_F - \mathbb{E} \mathsf{P}_F > t-W/2 \Big) 
= \mathbb{P}\Big((\mathsf{P}_F - \mathbb{E} \mathsf{P}_F)^3 > (t-W/2)^3 \Big) \\
					  &\leq \frac{\mu_3(\mathsf{P}_F)}{(t-W/2)^3} \leq \frac{W}{n^2(t-W/2)^3}
\end{align*}
Now we can use the last inequality in (\ref{P>F}) and find 
\begin{align*}
\mathbb{E}(\mathsf{P}_F^2 \mathbb{I}\{ \mathsf{P}_F > W\}) &= W^2 \mathbb{P}(\mathsf{P}_F > W) +  2 \int_{W}^\infty t \mathbb{P}(\mathsf{P}_F > t)  {\rm d} t \\
&\leq W^2 \frac{8 W}{n^2W^3} +  \frac{2W}{n^2} \int_{W}^\infty \frac{t}{(t-W/2)^3}  {\rm d} t \\
&= \frac{8}{n^2} + \frac{2W}{n^2} \Big[ \int_{W}^\infty  \frac{1}{(t-W/2)^2}  {\rm d} t + \int_{W}^\infty
  \frac{W/2}{(t-W/2)^3}  {\rm d} t\Big] \\
&= \frac{8}{n^2} + \frac{2W}{n^2} \Big[\frac{2}{W} + \frac{1}{W} \Big] = \frac{14}{n^2}
\end{align*}

We now provide a lower bound for the first term on the r.h.s. of \eqref{eq:minmax2}
\begin{align}
&\inf_{\hat{T}_n\leq W} \mathbb{E}_{\mathsf{P}} \mathbb{E}_{X |\mathsf{P}} [(M_{n}-\hat{T}_n)^2| \mathsf{P}] \nonumber \\
&\qquad = \inf_{\hat{T}_n\leq W} \mathbb{E}_m  \mathbb{E}_{\mathsf{P} | m} \Big\{ \mathbb{E}_{X |\mathsf{P}} [(M_{n}-\hat{T}_n)^2| \mathsf{P}]\ |\ m\ \Big\} \nonumber \\
& \qquad \geq \mathbb{E}_m  \inf_{\hat{T}_n\leq W} \mathbb{E}_{\mathsf{P} | m} \Big\{\mathbb{E}_{X |\mathsf{P}} [(M_{n}-\hat{T}_n)^2| \mathsf{P}]\ |\ m\ \Big\} \label{eq:prop_m}.
\end{align}

Let $m \in \mathbb{N}$, with the convention that the risk is $0$ when $m=0$. We will now lower bound inside the first expectation in the previous expression, using the fact that given $X$ and $m$, we know the posterior of $\mathsf{P}_F$.
\begin{align*}
&\inf_{\hat{T}_n\leq W} \mathbb{E}_{\mathsf{P} | m} \mathbb{E}_{X |\mathsf{P}} [(M_{n}-\hat{T}_n)^2| \mathsf{P}] \\
& \qquad =
\inf_{\hat{T}_n\leq W} \mathbb{E}_{X |m } \mathbb{E}_{\mathsf{P}| X,m} [(M_{n}-\hat{T}_n)^2|X,m] \\
&\qquad\geq \mathbb{E}_{X |m } \inf_{\hat{T}_n} \mathbb{E}_{\mathsf{P}| X,m} [(M_{n}-\hat{T}_n)^2| X,m] \\
&\qquad \geq  \mathbb{E}_{X |m } \Var_{\mathsf{P}| X,m} (M_{n}| X,m) 
 =  \mathbb{E}_{X |m } \Var_{\mathsf{P}| X,m} \left( \sum_{j = 1}^m p_j \mathbb{I}{\{ X_{n, j} = 0 \}} \Big| X,m \right) \\
&\qquad = \mathbb{E}_{X |m }\left[ \sum_{j = 1}^m \mathbb{I}{\{ X_{n, j} = 0 \}} \frac{a(n+b)}{(a+n+b)^2 (a+n+b+1)}  \right]
\end{align*}
where we used the fact that $\mathsf{p}_j | X_{n,j} =x $ is a beta random variable ${\rm Beta } (x+a, n-x+b)$. Then, 
\begin{align}
& \inf_{\hat{T}_n\leq W} \mathbb{E}_{\mathsf{P}|m} \mathbb{E}_{X |\mathsf{P}} [(M_{n}-\hat{T}_n)^2| \mathsf{P}] \nonumber \\
& \qquad\qquad\geq \sum_{j = 1}^m \mathbb{P} (X_{n, j} = 0 ) \frac{a(n+b)}{(a+n+b)^2 (a+n+b+1)} \nonumber \\
& \qquad\qquad=  \sum_{j = 1}^m \mathbb{E}_{\mathsf{p}_j} [(1-\mathsf{p}_j)^n] \frac{a(n+b)}{(a+n+b)^2 (a+n+b+1)}\nonumber\\
& \qquad\qquad =  \sum_{j = 1}^m \frac{\Gamma (b+n) \Gamma (a+b)}{\Gamma(b) \Gamma (a+b+n)} \frac{a(n+b)}{(a+n+b)^2 (a+n+b+1)} \label{eq:prop_zero}
\end{align}
where we observed that $(1-\mathsf{p}_j) \sim {\rm Beta} (b,a)$ to obtain \eqref{eq:prop_zero}.
Therefore we can derive the bound
\begin{align*}
\inf_{\hat{T}_n\leq W} \mathbb{E}_{\mathsf{P}|m} \mathbb{E}_{X |\mathsf{P}} [(M_{n}-\hat{T}_n)^2| \mathsf{P}]  & \geq  
m \frac{\Gamma (b+n) \Gamma (1+b)}{\Gamma(b) \Gamma (1+b+n)} \frac{n+b}{(n+b+1)^2 (n+b+2)}\\
& = \frac{m b}{(n+b+1)^2 (n+b+2)}\\
& = \frac{2 m n}{9n^2(3n+1)}  = \frac{2 m}{9n(3n+1)}
\end{align*}
Together with (\ref{eq:prop_m}), the previous inequality implies that
$$ \inf_{\hat{T}_n\leq W} \mathbb{E}_{\mathsf{P}} \mathbb{E}_{X |\mathsf{P}} [(M_{n}-\hat{T}_n)^2| \mathsf{P}] \geq \frac{2W}{9(3n+1)}, $$

Plugging in \eqref{eq:minmax2} the previous results,  we obtain
\begin{equation*}
R_n^\ast =\inf_{\hat{T}_n}\sup_{(p_j)_j \in \Pcr_{W}} \mathbb{E} [(M_{n}-\hat{T}_n)^2] \geq
 \frac{2W}{9(3n+1)} - \frac{14}{n^2},
\end{equation*}
which concludes the proof. Notice that here we are only interested in showing that the minimax rate is or order $\frac{W}{n}$. We could have obtained sharper bounds (with a better constant) by for instance taking $b = (1+\epsilon)n-1$ instead of $b=2n-1$.

\hfill \qed

\subsection{Details for the determiantion of \eqref{eq:W_concentration_right} and \eqref{eq:W_concentration_left}}  \label{AppProofConcW}

The concentration inequalities \eqref{eq:W_concentration_right}--\eqref{eq:W_concentration_left} can be proved my means of suitable bounds on the log-Laplace transform of $\hat{W}_n-W$, using the techniques of \cite{Bou(13)}.
To this end we recall that a random variable $Z$ is sub-Gaussian on the right tail (resp. left tail) with variance factor $v$ if for all $\lambda \geq 0$ (resp. $\lambda \leq 0$) one has
\begin{equation} \label{eq:subGaussian}
\log \mathbb{E} e^{\lambda (Z-\mathbb{E} Z)} \leq \frac{v \lambda^2}{2}.
\end{equation}
We will also say that a random variable $Z$ is sub-Gamma on the right tail with variance factor $v$ and scale parameter $c$ if
\[
\log \mathbb{E} e^{\lambda (Z-\mathbb{E} Z)} \leq \frac{\lambda^2v}{2(1-c\lambda)}, \quad \text{for any } \lambda \text{ such that }
0\leq \lambda\leq 1/c.
\]
Furthermore $Z$ is said to be sub-Gamma on the left tail with variance factor $v$ and scale parameter $c$, if $-Z$
is sub-Gamma on the right tail with variance factor $v$ and scale parameter $c$. These types of bounds on the log-Laplace transform imply exponential tail bounds, useful to find confidence intervals, for example if $Z$ is sub-Gamma on the right tail as above then
$\mathbb{P} (Z> \mathbb{E} Z +\sqrt{2vs}+cs)\leq e^{-s}$ for any $s \geq 0$. Analogously if $Z$ is sub-Gaussian on the left tail (i.e. \eqref{eq:subGaussian} holds true for all $\lambda \leq 0$), then 
$\mathbb{P} (Z < \mathbb{E} Z - \sqrt{2vs}) \leq e^{-s}$ for any $s \geq 0$. See \cite{Bou(13)} for additional details on this subject. In the next proposition we prove that $\hat{W}_n$ is sub-Gamma on the right tail and sub-Gaussian on the left tail, providing suitable bounds on the log-Laplace transform. As just explained these bounds on the log-Laplace immediately implies the exponential tail bounds \eqref{eq:W_concentration_right} and \eqref{eq:W_concentration_left}.
\begin{prop}\label{prop:W_logLaplace} 
The random variable $\hat{W}_n := \sum_{j \geq 1} X_{n,j}/n$ is an unbiased estimator of $W$. In addition
$\hat{W}_n $ is  sub-Gamma on the right tail with variance factor
$w_n^+ := 2 W/n$ and scale factor $1/n$, i.e. for any $0< \lambda < n$
\begin{equation}
\label{eq:right_tail_Wnr}
\log \mathbb{E} \left( e^{\lambda (\hat{W}_n - W)} \right) \leq 
\frac{w_n^+\lambda^2}{2(1-\lambda/n)};
\end{equation}
on the left tail, $\hat{W}_n$ is sub-Gaussian with variance factor $w_n^-= W/n$, namely for all $\lambda \geq 0$ we have
\begin{equation}
\label{eq:left_tail_Wnr}
\log \mathbb{E} \left( e^{-\lambda (\hat{W}_n - W)} \right) \leq 
\frac{\lambda^2 w_n^-}{2}.
\end{equation} 
\end{prop}
\begin{proof}
One can immediately show that $\mathbb{E}(\hat{W}_n) = W$.
We first prove \eqref{eq:right_tail_Wnr}, that is to say $\hat{W}_n - W$ is sub-Gamma on the right tail, indeed, by the  independence of the $X_{n,j}$'s, we get
\begin{align*}
\log \mathbb{E} \left( e^{\lambda (\hat{W}_n - W)} \right) & = \sum_{j \geq 1} \log \mathbb{E}
\left( e^{\lambda/n (X_{n,j}-\mathbb{E} X_{n,j})} \right)\\
& = \sum_{j \geq 1} \left( -\lambda p_j +\log \mathbb{E} e^{\lambda/n X_{n,j}}  \right)\\
& = \sum_{j \geq 1} \left( -\lambda p_j +n \log (1-p_j +p_j e^{\lambda/n})  \right)
\end{align*} 
having observed that $X_{n,j}$ is a Binomial random variable with parameters $n$ and $p_j$. An application of the inequality 
$\log (z) \leq z-1$, for $z >0$, yields
\begin{equation} \label{eq:W_right1}
\log \mathbb{E} \left( e^{\lambda (\hat{W}_n - W)} \right)  \leq n \sum_{j \geq 1} p_j (e^{\lambda/n}- 1-\lambda/n)
= n \phi(\lambda/n) \mathbb{E} \hat{W}_n
\end{equation}
where $\phi (t) = e^t -1-t$. By a series expansion of the exponential function, we obtain 
\begin{align*}
\log \mathbb{E} \left( e^{\lambda (\hat{W}_n - W)} \right) & \leq n \sum_{k \geq 2} \frac{1}{k!}\left(\frac{\lambda}{n}
\right)^k \mathbb{E} \hat{W}_n \leq  n \sum_{k \geq 2} \left(\frac{\lambda}{n}
\right)^k \mathbb{E} \hat{W}_n \\
& =n \cdot\frac{\lambda^2}{n^2} \cdot\frac{1}{1-\lambda/n} \mathbb{E} \hat{W}_n = \frac{w_n^+\lambda^2}{2(1-\lambda/n)}
\end{align*} 
and then \eqref{eq:right_tail_Wnr} follows for any $\lambda \in (0,n)$. The proof of \eqref{eq:left_tail_Wnr} goes along similar lines, indeed proceeding as before one can prove that for any $\lambda  \geq 0$
\begin{equation*}
\log \mathbb{E} \left( e^{-\lambda (\hat{W}_n - W)} \right) \leq n \phi (-\lambda/n) \mathbb{E} \hat{W}_n
\end{equation*}
and then \eqref{eq:left_tail_Wnr} follows by observing that $\phi (-\lambda/n)  \leq \lambda^2/(2n^2) $.
\end{proof}

\subsection{Proof of Theorem \ref{thm2} and related results} \label{AppProofThm2}

In this section we focus on  Theorem \ref{thm2}. First of all we introduce some  preliminary results which are useful to prove the theorem and  are interesting in their own right. 
In the next proposition we derive concentration inequalities for both $K_{n,r}$ and $K_n$.

\begin{prop} \label{prop1}
For any $n \geq 1$,  $r \geq 1$ and $\delta \in  (0, 1)$ we have
\begin{equation} \label{eq:concentration_Knr}
\mathbb{P} \left(|K_{n, r} -\mathbb{E} (K_{n,r})| \leq \frac{2}{3}
\log \left( \frac{2}{\delta} \right) +\sqrt{2 \mathbb{E} (K_{n,r}) \log \left( \frac{2}{\delta} \right)} \right) \geq 1- \delta
\end{equation}
and
\begin{equation} \label{eq:concentration_Knr_bar}
\mathbb{P} \left(|K_{n} -\mathbb{E} (K_{n})| \leq \frac{2}{3}
\log \left( \frac{2}{\delta} \right) +\sqrt{2 \mathbb{E} (K_n) \log \left( \frac{2}{\delta} \right)} \right) \geq 1- \delta .
\end{equation}
\end{prop}

\bigskip
\bigskip

\begin{proof}
Let us  focus on the derivation of \eqref{eq:concentration_Knr}.
First of all we will determine a variance bound for the random variable $K_{n,r}$, which will be employed to derive a corresponding bound on the log-Laplace transform of $(K_{n, r}-\mathbb{E} (K_{n, r}) )$. Then the result will follow by a suitable  application of the Bernstein inequality.\\
Thanks to the independence of the random variables $X_{n,j}$'s, the variance of $K_{n,r} $ may be bounded as follows
\begin{equation} \label{eq:var_Knr}
\Var (K_{n,r})  
= \sum_{j \geq 1} \Var \left( \mathbb{I}{\{ X_{n, j} = r  \}} \right)\leq \sum_{j \geq 1} \mathbb{E} \left(\mathbb{I}{\{ X_{n, j} = r  \}}   \right) = \mathbb{E} (K_{n, r}).
\end{equation}
We now establish the bound on the log-Laplace. 
Since $K_{n, r}$ is the sum of independent random variables, for any $t \in \R $ we can write:
\begin{align*}
\log \mathbb{E} (e^{t(K_{n, r}-\mathbb{E} (K_{n, r}) )}) & =
\sum_{j=1}^{\infty} \log \mathbb{E} \exp \left\{ (\mathbb{I}{\{ X_{ n,j} = r \}} - \mathbb{E}
\mathbb{I}{\{ X_{ n,j} =r \}}  ) \right\} \\
& \leq \phi (|t|) \sum_{j=1}^{\infty} \Var (\mathbb{I}{\{ X_{ n,j} = r \}} ) 
= \phi (|t|) \Var (K_{n,r}) \\
& \stackrel{\eqref{eq:var_Knr}}{\leq} \phi (|t|) \mathbb{E} (K_{n,r})
\end{align*}
where we have implicitly defined the function $\phi (t)= e^t-1-t$ and we have applied the Bennett inequality. The validity of the previous bound on the log-Laplace implies that
for any $x \geq 0$
\[
\mathbb{P} (|K_{n,r}-\mathbb{E} (K_{n,r})| \geq x) \leq 2\exp \left\{ -\frac{x^2}{2(\mathbb{E} (K_{n,r})+x/3)}
\right\}
\]
thanks to  the Bernstein inequality. If one defines
$
s:= x^2/(2(\mathbb{E} (K_{n,r}) +x/3))
$
the previous inequality boils down to the following one
\[
\mathbb{P} \left( |K_{n,r}-\mathbb{E} (K_{n,r})|  \geq \frac{2}{3}s +\sqrt{2s \mathbb{E} (K_{n,r})  }\right)
\leq 2e^{-s}
\]
which in turn implies the validity of \eqref{eq:concentration_Knr}, as a consequence of the change 
of variable $2e^{-s}=\delta$ and the elementary inequality $\sqrt{a+b}\leq \sqrt{a}
+\sqrt{b}$, valid for any positive real numbers $a,b$. The proof of \eqref{eq:concentration_Knr_bar} goes along similar lines, having observed that $\Var (K_n) \leq \mathbb{E} (K_n)$.
\end{proof}

\bigskip
\bigskip

The next Remark is a simple consequence of Proposition \ref{prop1} and will be used in the derivation of the confidence interval for $\hat{M}_n$ to upper bound the expected values $\mathbb{E}(K_{n,r})$ and $\mathbb{E} (K_n)$ in terms of the two statistics  $K_{n,r}$ and $K_n$ respectively.

\begin{remark} \label{remark1}
The concentration inequality for $K_{n,r}$ \eqref{eq:concentration_Knr} implies that with  probability bigger than $1-\delta$ 
\[
K_{n,r} \geq \mathbb{E} (K_{n,r}) -\frac{2}{3} \log \left( 1/\delta \right) -\sqrt{2 \log \left( 1/\delta \right) \mathbb{E} (K_{n,r})}.
\]
Solving the previous inequality  with respect to $\sqrt{\mathbb{E} (K_{n,r})}$, we can conclude that with probability at least $1-\delta$
\[
\sqrt{\mathbb{E} (K_{n,r})} \leq \sqrt{\frac{1}{2} \log \left( 1/\delta \right)}+
\sqrt{\frac{7}{6} \log \left( 1/\delta \right)+ K_{n,r}}.
\]
Analogously we can employ the concentration inequality for $K_n$ (see \eqref{eq:concentration_Knr_bar}) to state that the event
\[
\sqrt{\mathbb{E} (K_{n})} \leq \sqrt{\frac{1}{2} \log \left( 1/\delta \right)}+
\sqrt{\frac{7}{6} \log \left( 1/\delta \right)+ K_{n}}
\]
has probability at least $1-\delta$.
\end{remark}
\bigskip

The confidence interval we suggest is mainly based on a concentration inequality for $\hat{M}_n-M_n$. Such concentration result is proved my means of suitable bounds on the log-Laplace transform of $\hat{M}_n-M_n$, as in Section \ref{AppProofConcW}, more precisely in Proposition \ref{prop2} we will prove that  $\hat{M}_n-M_n$ is a sub-Gamma random variable on the right and the left tails (see \cite{Bou(13)}). 
In order to provide the bounds on the log-Laplace transform of $\hat{M}_n-M_n$, the following lemma is required. \\

\begin{lemma}\label{lemma1} 
For any $n \geq 1$, we have
\begin{equation}
\label{eq:bias_Gnr_2}
\mathbb{E} (\hat{M}_n-M_{n}) \leq \frac{2}{n(n-1)}\mathbb{E} (K_{n, 2}),
\end{equation}
and
\begin{equation}
\label{eq:var_Gnr_2}
\Var (\hat{M}_n-M_{n} ) \leq \frac{1}{n^2} \mathbb{E} (K_{n, 1})
+\frac{2}{n(n-1) } \mathbb{E} (K_{n,2}).
\end{equation} 
\end{lemma}
 \bigskip

\begin{proof} 
First of all, we focus on the proof of \eqref{eq:bias_Gnr_2}, straightforward calculations lead to
\begin{align*}
\mathbb{E} (\hat{M}_n-M_{n})  & = \frac{1}{n} \sum_{j \geq 1} {n \choose 1} 
p_j (1-p_j)^{n-1} -\sum_{j \geq 1}  p_j (1-p_j)^{n}\\
& = \sum_{j \geq 1}  p_j (1-p_j)^{n-1}
-  \sum_{j \geq 1}p_j(1-p_j)^{n}\\
& =  \sum_{j \geq 1} p_j^{2} (1-p_j)^{n-1}
\leq \sum_{j \geq 1}  p_j^{2} (1-p_j)^{n-2}\\
& = \frac{2}{n(n-1)}\sum_{j \geq 1} {n \choose 2} p_j^{2} (1-p_j)^{n-2}
= \frac{2}{n(n-1)} \mathbb{E} (K_{n,2})
\end{align*}
where we have observed  that $(1-p_j)^{n+1-2} \leq (1-p_j)^{n-2}$, which is legitimated by the fact that 
$n$ is a sample size and can be assumed sufficiently large.\\
Second, we focus on the proof of the variance bound \eqref{eq:var_Gnr_2}. Exploiting the independence of the random variables $X_{n, j}$'s, we get
\begin{align*}
\Var (\hat{M}_n-M_{n}) & = \sum_{j \geq 1} \Var \left( \frac{1}{n}
\mathbb{I}{\{ X_{n, j} = 1 \}} -p_j \mathbb{I}{\{ X_{n, j}=0 \}} \right) \\
& \leq \sum_{j \geq 1} \mathbb{E} \left(  \frac{1}{n}
\mathbb{I}{\{ X_{n, j} = 1 \}} -p_j \mathbb{I}{\{ X_{n, j}=0 \}}\right)^2.
\end{align*}
We now observe that the events  $\{ X_{n, j}=1 \}$ and $\{ X_{n, j}=0 \}$ are incompatible, hence
\begin{align*}
\Var (\hat{M}_n-M_{n}) & \leq  \sum_{j \geq 1} \mathbb{E} \left(  \frac{1}{n}
\mathbb{I}{\{ X_{n, j} = 1 \}} -p_j \mathbb{I}{\{ X_{n, j}=0 \}}\right)^2\\
& =  \sum_{j \geq 1} \left( \left( \frac{1}{n} \right)^2 \mathbb{P} (X_{n, j}= 1)
+ p_j^2 \mathbb{P} (X_{n, j} = 0)\right)\\
& =  \frac{1}{n^2} \sum_{j \geq 1} \mathbb{P} (X_{n, j}= 1)
+ \sum_{j \geq 1} p_j^2 \mathbb{P} (X_{n, j} = 0)\\
& = \frac{1}{n^2}  \mathbb{E} (K_{n, 1}) +\sum_{j \geq 1} 
p_j^{2} (1-p_j)^{n} \leq   \frac{1}{n^2}  \mathbb{E}(K_{n, 1})  +\sum_{j \geq 1} 
p_j^{2} (1-p_j)^{n-2}\\
& = \frac{1}{n^2}  \mathbb{E}(K_{n, 1})  +\frac{2}{n(n-1)} \mathbb{E} (K_{n,2})
\end{align*}
and the bound on the variance now follows. 

\end{proof}  \bigskip

To fix the notation, it is worth to define the quantities
 \begin{equation*}
 \mathbb{E} [K(n)] = \sum_{j \geq 1} (1-e^{-np_j}) , \quad \mathbb{E} [K_k(n)] =\sum_{j \geq 1} e^{-np_j} \frac{(np_j)^k}{k!}
\end{equation*}  
for any $k \geq 1$. It can be easily seen that $ \mathbb{E} [K_k(n)]\leq \mathbb{E} [K(n)]$, and in addition $\mathbb{E} [K(n)]$ is an increasing function of $n$.
We are now ready to prove  that $\hat{M}_n-M_n$ is sub-Gamma on its tails. \\ \bigskip

\begin{prop}\label{prop2} 
The variable $D_n= \hat{M}_n-M_n$ is sub--gamma on the right tail with variance factor
$v_{n}^+=2\{\mathbb{E} (K_{n,1})+2n^2\mathbb{E}(K_{n,2})/[n\cdot (n-1)]\}/n^2$ and scale factor $1/n$, i.e. for any $0< \lambda < n$
\begin{equation}
\label{eq:right_tail_Ynr}
\log \mathbb{E} \left( e^{\lambda [D_n-\mathbb{E} (D_n)]} \right) \leq 
\frac{v_{n}^+ \lambda^2}{2(1-\lambda/n)}.
\end{equation}
On the left tail, $D_n$ is again sub--gamma with variance factor $v_n^-= 4 \mathbb{E} [K(n)]/(n-1)^2$
and scale parameter $1/(n-1)$, i.e. for any $0 < \lambda < (n-1)$
\begin{equation}
\label{eq:left_tail_Ynr}
\log \mathbb{E} \left( e^{-\lambda [D_n-\mathbb{E} (D_n)]} \right) \leq 
\frac{v_{n}^- \lambda^2}{2[1-\lambda/(n-1)]}.
\end{equation} 
\end{prop} \bigskip

\begin{proof} 
First we prove \eqref{eq:right_tail_Ynr}, observe that when $\lambda >0$
\begin{multline*}
\log \mathbb{E} \left( e^{\lambda [D_n-\mathbb{E} (D_n)]} \right) \\ =
\sum_{j \geq 1}\log \mathbb{E} \exp \Big\{ \frac{\lambda}{n} \Big[ \mathbb{I}{\{ X_{n, j}=1 \}} -np_j \mathbb{I}{\{ X_{n, j}=0 \}}   -\mathbb{E} \Big(  \mathbb{I}{\{ X_{n, j}=1 \}} -np_j \mathbb{I}{\{ X_{n, j}=0 \}} \Big) \Big] \Big\}.
\end{multline*}
For the sake of simplifying notation, define
$H_j : = \mathbb{I}{\{ X_{n, j}=1 \}} -np_j \mathbb{I}{\{ X_{n, j}=0 \}} \leq 1$, hence we have
\begin{align*}
\log \mathbb{E} \left( e^{\lambda [D_n-\mathbb{E} (D_n)]} \right)& =
\sum_{j \geq 1} \log \mathbb{E} \exp \left\{ \frac{\lambda}{n} (H_j -\mathbb{E} H_j)  \right\}\\
& = \sum_{j \geq 1} \left( \log \mathbb{E} [e^{(\lambda/n)H_j}] -\frac{\lambda}{n} \mathbb{E} H_j \right) \\
& \leq \sum_{j \geq 1} \left(  \mathbb{E} [e^{(\lambda/n)H_j} -1]- \frac{\lambda}{n} \mathbb{E} H_j\right)
\end{align*}
where we used the inequality $\log(z)\leq z-1$ for any positive $z$. Now, observing that
the function $f(z):= \frac{e^z-1-z}{z^2}$ is increasing and that $H_j \leq 1$, we get
\begin{align*}
& \log \mathbb{E} \left( e^{\lambda [D_n-\mathbb{E} (D_n)]} \right) 
\leq \sum_{j \geq 1} \mathbb{E} \left[  e^{(\lambda/n)H_j} -1- \frac{\lambda}{n} H_j  \right]\\
& \qquad = \sum_{j \geq 1}  \mathbb{E} \left\{ \frac{ e^{(\lambda/n)H_j} -1- \frac{\lambda}{n} H_j }{
[(\lambda/n)H_j]^2} [(\lambda/n)H_j]^2 \right\} \\
& \qquad \leq \sum_{j \geq 1} \mathbb{E} (H_j^2) (e^{\lambda/n}-1-\lambda/n)\\
& \qquad = \phi (\lambda/n) \sum_{j \geq 1} \mathbb{E} \left( \mathbb{I}{\{ X_{n, j}=1 \}} -np_j \mathbb{I}{\{ X_{n, j}=0 \}} \right)^2 \\
& \qquad \leq \sum_{k \geq 2} \left( \frac{\lambda}{n} \right)^k  \sum_{j \geq 1} \mathbb{E} \left( \mathbb{I}{\{ X_{n, j}=1 \}} -np_j \mathbb{I}{\{ X_{n, j}=0 \}} \right)^2\\
& \qquad  =  \frac{\lambda^2}{1-\lambda/n} \sum_{j \geq 1} \mathbb{E} \left( \frac{1}{n}\mathbb{I}{\{ X_{n, j}=1 \}} -p_j \mathbb{I}{\{ X_{n, j}=0 \}} \right)^2
\end{align*}
now one can proceed along similar arguments as in the proof of \eqref{eq:var_Gnr_2} to obtain
\begin{equation*}
\log \mathbb{E} \left( e^{\lambda [D_n-\mathbb{E} (D_n)]} \right) \leq 
\frac{\lambda^2}{1-\lambda/n} \left( \frac{1}{n^2}\mathbb{E} (K_{n,1}) +\frac{2}{n(n-1)} \mathbb{E} (K_{n,2}) \right)
\end{equation*}
hence \eqref{eq:right_tail_Ynr} now follows.\\
Finally we prove \eqref{eq:left_tail_Ynr}, analogous calculations as before show that
\begin{equation*}
\log \mathbb{E} \left( e^{-\lambda [D_n-\mathbb{E} (D_n)]} \right)=
\sum_{j \geq 1} \log \mathbb{E} \exp \left\{  \lambda p_j (W_j -\mathbb{E} W_j) \right\}
\end{equation*}
having defined $W_j := \mathbb{I}{\{ X_{n, j}=0 \}}- \frac{1}{np_j} \mathbb{I}{\{ X_{n, j} =1 \}} \leq 1$. Thanks to the inequality $\log (z) \leq z-1$, which is valid for any positive $z$, and applying similar considerations as those used in the first part of the proof, we have
\begin{align*}
 &\log \mathbb{E} \left( e^{-\lambda [D_n-\mathbb{E} (D_n)]} \right) =
 \sum_{j \geq 1 }  \left\{ \log [E (e^{\lambda p_j W_j})]-\lambda p_j E W_j\right\}\\
 & \qquad \leq \sum_{j \geq 1}  \mathbb{E} [e^{\lambda p_j W_j} -1-\lambda p_j W_j]
 =\sum_{j \geq 1}  \mathbb{E} \left[ \frac{ e^{\lambda p_j W_j} -1-\lambda p_j W_j}{(\lambda p_j W_j)^2}
 (\lambda p_j W_j)^2 \right]\\
 &\qquad \leq\sum_{j \geq 1}  \mathbb{E} \left[ \frac{ e^{\lambda p_j} -1-\lambda p_j}{(\lambda p_j)^2}
 (\lambda p_j W_j)^2 \right]  = \sum_{j \geq 1} \mathbb{E} W_j^2 \phi (\lambda p_j)\\ 
 & \qquad = \sum_{j \geq 1} \phi (\lambda p_j) \mathbb{E} \left( \mathbb{I}{\{ X_{n, j} = 0 \}}-
 \frac{1}{np_j} \mathbb{I}{\{ X_{n,j} = 1 \}} \right)^2 \\
& \qquad \leq \sum_{j \geq 1} \phi (\lambda p_j) \mathbb{E} \left( \mathbb{I}{\{ X_{n, j} = 0 \}}+
 \frac{1}{n^2p_j^2} \mathbb{I}{\{ X_{n,j} = 1 \}} \right)\\
   & \qquad = \sum_{j \geq 1} \phi (\lambda p_j) 
   \left[ (1-p_j)^n +\frac{1}{np_j} (1-p_j)^{n-1} \right] \\
   & \qquad   = \sum_{j \geq 1} \sum_{k \geq 2} \frac{(\lambda p_j)^k}{k!}
   \left[ (1-p_j)^n +\frac{1}{np_j} (1-p_j)^{n-1} \right]\\
&\qquad \leq \sum_{k \geq 2}  \lambda^k \sum_{j \geq 1} \left( \frac{p_j^k}{k!} e^{-np_j} + 
\frac{p_j^{k-1}}{n k!} e^{-(n-1)p_j} \right),
\end{align*}
where we have observed that $(1-p)^n \leq e^{-np}$. Recalling the definition of $K_k (n)$  and that $\mathbb{E} [K_k(n)]\leq \mathbb{E} [K(n)]$, we get
\begin{align*}
 &\log \mathbb{E} \left( e^{-\lambda [D_n-\mathbb{E} (D_n)]} \right)\\
 & \qquad \leq 
 \sum_{k \geq 2} \lambda^k \left\{ \frac{1}{n^k} \sum_{j\geq 1}  \frac{(np_j)^k}{k!}e^{-np_j} +
 \frac{1}{kn(n-1)^{k-1}}\sum_{j \geq 1} \frac{[(n-1)p_j]^{k-1}}{(k-1)!} e^{-(n-1)p_j}\right\}\\
 & \qquad = \sum_{k \geq 2} \lambda^k \left\{ \frac{1}{n^k} \mathbb{E} [K_k(n)] +\frac{1}{kn(n-1)^{k-1}}
 \mathbb{E} [K_{k-1} (n-1)] \right\}\\
  & \qquad \leq \sum_{k \geq 2} \left( \frac{\lambda}{n-1} \right)^k
  \left\{  \mathbb{E} [K_k(n)] +\frac{1}{k} \mathbb{E} [K_{k-1} (n-1)] \right\}\\
    & \qquad \leq \{ \mathbb{E} [K(n)]+\mathbb{E} [K(n-1)]\}\sum_{k \geq 2} \left( \frac{\lambda}{n-1} \right)^k
     \leq \frac{2\mathbb{E} [K(n)]}{(n-1)^2} \frac{\lambda^2}{1-\lambda/(n-1)} ,
\end{align*}
hence \eqref{eq:left_tail_Ynr} has been now established. 
\end{proof} 
 \bigskip

Finally we prove Theorem \ref{thm2}. \\

\begin{proof}[Proof of Theorem  \ref{thm2}]
The result is a consequence of the two bounds that we are now going to prove separately:
\begin{align}
\label{eq:bound_right}
&  \mathbb{P} \Big( M_n \geq \hat{M}_n- \hat{m}_n -\sqrt{2 \hat{v}^+_n \log (1/\delta)} -\log (1/\delta)/n \Big)
\geq 1-4\delta \\
\label{eq:bound_left}
& \mathbb{P} \Big( M_n
\leq \hat{M}_n +\sqrt{2\hat{v}^-_n \log (1/\delta)} +\log(1/\delta)/(n-1) \Big)  \geq 1-2\delta .
\end{align}
where $\hat{m}_n, \hat{v}^-_n$ and $\hat{v}^+_n$ are suitable quantities that will be defined in the proof.
First we discuss how to determine \eqref{eq:bound_right}, we remind that $D_n= \hat{M}_n -M_n$ is a sub--Gamma random variable, as shown in Proposition \ref{prop2}, hence the following holds (see \cite{Ben17})
\begin{equation*}
\mathbb{P} (D_n > \mathbb{E} (D_n) +\sqrt{2 v_n^+ s} +s/n) \leq e^{-s}
\end{equation*}
for any $s \geq 0$. Putting $e^{-s}=\delta$, we obtain
\begin{equation} \label{eq:r1}
\mathbb{P} (D_n \leq  \mathbb{E} (D_n) +\sqrt{2 v_n^+ \log (1/\delta)} +\log (1/\delta)/n) \geq 1-\delta.
\end{equation}
On the basis of \eqref{eq:bias_Gnr_2}, we can upper bound  $\mathbb{E} (D_n)$ in \eqref{eq:r1} as
\begin{equation*}
\mathbb{E} (D_n) \leq \frac{2}{n(n-1)}\mathbb{E} (K_{n, 2}),
\end{equation*}
and, as a consequence of Remark A.3.2, the following inequality is true with probability bigger than $1-\delta$
\begin{equation}
\label{eq:Y_hat}
\mathbb{E} (D_n) \leq \frac{2}{(n-1)n}\left( \sqrt{\frac{1}{2}\log(1/\delta)}+\sqrt{\frac{7}{6}\log (1/\delta)+
K_{n,2}} \right)^2 =: \hat{m}_n.
\end{equation}
Remark A.3.2 also  applies to upper bound the expected values in $v_n^+$ with the corresponding empirical values, indeed the following bound holds true
\begin{equation}
\label{eq:V_hat}
\begin{split}
&v_n^+ \leq \frac{2}{n^2} \left( \sqrt{\frac{1}{2}\log(1/\delta)}+ \sqrt{\frac{7}{6} \log (1/\delta) +K_{n,1}}  \right)^2 \\
&\qquad\qquad+ \frac{4}{n(n-1)} \left(   \sqrt{\frac{1}{2}\log(1/\delta)} + \sqrt{\frac{7}{6} \log (1/\delta)+K_{n,2}}\right)^2 =: \hat{v}_n^+
\end{split}
\end{equation}
with probability bigger than $1-2\delta$. Using the inequalities \eqref{eq:Y_hat}--\eqref{eq:V_hat} in 
\eqref{eq:r1} we can conclude
\begin{equation*} 
\mathbb{P} (D_n \leq  \hat{m}_n +\sqrt{2 \hat{v}_n^+ \log (1/\delta)} +\log (1/\delta)/n) \geq 1-4\delta,
\end{equation*}
which gives \eqref{eq:bound_right}.\\
Now we focus on the derivation of \eqref{eq:bound_left}, first of all recall that by Proposition \ref{prop2}, $-D_n$ is a sub--Gamma random variable, hence 
\begin{equation*}
\mathbb{P} (-D_n > -\mathbb{E} [D_n] +\sqrt{2 v_n^- s} +s/(n-1)) \leq e^{-s}
\end{equation*}
for any $s \geq 0$. Define $e^{-s}=\delta$, then the previous bound implies that
\begin{equation*}
\mathbb{P} (D_n \geq  \mathbb{E} [D_n] -\sqrt{2 v_n^- \log(1/\delta)} -\log(1/\delta)/(n-1)) \geq 1-\delta
\end{equation*}
since $\mathbb{E} [D_n] \geq 0$, we get
\begin{equation}
\label{eq:l1}
\mathbb{P} (D_n \geq  -\sqrt{2 v_n^- \log(1/\delta)} -\log(1/\delta)/(n-1)) \geq 1-\delta.
\end{equation}
One can now apply the bound contained in \cite[Lemma 1]{Gne(07)}
\[
|\mathbb{E} [K(n)]-\mathbb{E} (K_n)| \leq \frac{2}{n} \mathbb{E} [K_2(n)]\leq \frac{2}{n}\mathbb{E} [K(n)],
\]
which entails that 
$\mathbb{E} [K(n)]\leq \mathbb{E} (K_n)/(1-2/n)$, and hence
\[
v_n^- \leq  \frac{4}{(n-1)^2 (1-2/n)} \mathbb{E} (K_n). 
\]
Thanks to Remark A.3.2, with probability bigger than $1-\delta$, the following holds
\begin{equation}
\label{eq:V-_hat}
v_n^- \leq \frac{4}{(n-1)^2 (1-2/n)}  \left(  \sqrt{\frac{1}{2}\log (1/\delta)} +\sqrt{ \frac{7}{6}
\log (1/\delta) +K_n}\right)^2 =: \hat{v}_n^-.
\end{equation}
Using \eqref{eq:V-_hat} in \eqref{eq:l1}, we get
\begin{equation*}
\mathbb{P} (D_n \geq  -\sqrt{2 \hat{v}_n^- \log(1/\delta)} -\log(1/\delta)/(n-1)) \geq 1-2\delta,
\end{equation*}
which in turn implies \eqref{eq:bound_left}. Putting together \eqref{eq:bound_right}--\eqref{eq:bound_left}, the thesis now follows. 
\end{proof}

\section*{Acknowledgements}

Federico Camerlenghi and Stefano Favaro received funding from the European Research Council (ERC) under the European Union's Horizon 2020 research and innovation programme under grant agreement No 817257. Federico Camerlenghi and Stefano Favaro gratefully acknowledge the financial support from the Italian Ministry of Education, University and Research (MIUR), ``Dipartimenti di Eccellenza" grant 2018--2022.


\begin{thebibliography}{9}



\bibitem[{Auton et al.(2015)}]{Aut(15)}
\textsc{Auton, A. et al.} (2015).
\newblock {A global reference for human genetic variation.}
\newblock \textit{Nature} \textbf{526}, 68--74.


\bibitem[{Ben-Hamou et al. (2017)}]{Ben17}
\textsc{Ben-Hamou, A., Boucheron, S. and Ohannessian, M.I.} (2017).
\newblock {Concentration inequalities in the infinite urn scheme for occupancy counts and the missing mass, with applications.}
\newblock \textit{Bernoulli} \textbf{23}, 249--287.

\bibitem[{Boucheron et al. (2013)}]{Bou(13)}
\textsc{Boucheron, S., Lugosi, G. and  Massart, P.} (2013).
\newblock{ \textit{Concentration inequalities.}}
\newblock Oxford University Press.

\bibitem[{Chao and Colwell(2017)}]{Cha(17)}
\textsc{Chao, A. and Colwell, R.K.} (2017). Thirty years of progeny from Chao's inequality: estimating and comparing richness with incidence data and incomplete sampling. \textit{Statistics and Operation Research Transactions}, \textbf{41}, 3--54.

\bibitem[{Chao et al.(2014)}]{Cha(14)}
\textsc{Chao, A., Gotelli, N.J., Hsieh, T.C., Sander, E.L., Ma, K.H., Colwell, R.K. and Ellison, A.M.} (2014).
\newblock {Rarefaction and extrapolation with Hill numbers: a framework for sampling and estimation in species diversity studies.}
\newblock \textit{Ecological Monographs} \textbf{84}, 45--67.


\bibitem[{Colwell et al.(2012)}]{Col(12)}
\textsc{Colwell, R., Chao, A., Gotelli, N.J., Lin, S., Mao, C.X., Chazdon, R.L. and Longino, J.T.} (2012).
\newblock {Models and estimators linking individual-based and sample-based rarefaction, extrapolation and comparison of assemblages.}
\newblock \textit{Journal of Plant Ecology} \textbf{5}, 3--21.

\bibitem[{Daley and Vere-Jones(2008)}]{Dal(08)}
\textsc{Daley, D.J. and Vere-Jones, D.} (2008).
\newblock{\textit{An introduction to the theory of point processes. Vol. II}}. Springer, New York.

\bibitem[{Efron(1987)}]{Efr(87)}
\textsc{Efron, B.} (1987).
\newblock \textit{The jackknife, the bootstrap, and other resampling plans.}
\newblock  CBMS-NSF Regional Conference Series in Applied Mathematics, Society for Industrial and Applied Mathematics.

\bibitem[{Efron and Morris(1973)}]{Efr(73)}
\textsc{Efron, B. and Morris, C} (1973).
\newblock {Stein's estimation rule and its competitors - an empirical Bayes approach.}
\newblock \textit{Journal of the American Statistical Association} \textbf{68}, 117--130.

\bibitem[{Good(1953)}]{Goo(53)}
\textsc{Good, I.J.} (1953).
\newblock {On the population frequencies of species and the estimation of population parameters.}
\newblock \textit{Biometrika} \textbf{40}, 237--264.

\bibitem[{G\"or\"ur et al.(2006)}]{Gor(06)}
\textsc{G\"or\"ur, D., J\"akel, F. and Rasmussen, C.E.} (2006).
\newblock {A choice model with infinitely many latent features.}
\newblock \textit{23rd International Conference on Machine Learning.}

\bibitem[Gnedin et al. (2007)]{Gne(07)}  
\textsc{Gnedin, A., Hansen, B. and Pitman, J.} (2007). 
\newblock {Notes on the occupancy problem with infinitely many boxes: general asymptotics and power laws.} \newblock \textit{Probability Surveys}, \textbf{4}, 146--171.

\bibitem[{Gravel(2014)}]{Gra(14)}
\textsc{Gravel, S} (2014).
\newblock {Predicting discovery rates of genomic features}
\newblock \textit{Genetics} \textbf{197}, 601--610.

\bibitem[{Hjort(1990)}]{Hjo(90)}
\textsc{Hjort, N.} (1990).
\newblock {Nonparametric Bayes estimators based on Beta processes in models for life history data.}
\newblock \textit{The Annals of Statistics} \textbf{18}, 1259--1294.

\bibitem[{Ionita-Laza et al.(2009)}]{Ion(09)}
\textsc{Ionita-Laza, I., Lange, C. and Laird, N.M.} (2009).
\newblock {Estimating the number of unseen variants in the human genome.}
\newblock \textit{Proceeding of the National Academy of Sciences} \textbf{106}, 5008--5013.

\bibitem[James(2017)]{Jam(17)}
\textsc{James, L.F.} (2017).
\newblock{Bayesian Poisson calculus for latent feature modeling via generalized Indian buffet process priors.}
\newblock \textit{The Annals of Statistics} \textbf{45}, 2016--2045.

\bibitem[Karlin(1967)]{Kar(67)}  
\textsc{Karlin, S.} (1967). 
\newblock{Central limit theorems for certain infinite urn schemes.}
\newblock \textit{Journal of Mathematics and Mechanics}, \textbf{17}, 373--401.

\bibitem[{Meeds et al.(2007)}]{Mee(07)}
\textsc{Meeds, E., Ghahramani, Z., Neal, R. and Rowies, S.T.} (2007).
\newblock {Modeling dyadic data with binary latent factors.}
\newblock \textit{Advances in Neural Information Processing Systems.}

\bibitem[{Miller et al.(2006)}]{Mil(10)}
\textsc{Miller, K.T., Griffiths, T.L. and Jordan, M.I.} (2010).
\newblock {Nonparametric latent feature models for link predictions.}
\newblock \textit{Advances in Neural Information Processing Systems.}

\bibitem[{Navarro and Griffiths(2007)}]{Nav(07)}
\textsc{Navarro, D.J. and Griffiths, T.L.} (2010).
\newblock { A nonparametric Bayesian model for inferring features from similarity judgments.}
\newblock \textit{Advances in Neural Information Processing Systems.}

\bibitem[{Quenouille(1956)}]{Que(56)}
\textsc{Quenouille, M.H.} (1956).
\newblock {Notes on bias in estimation.}
\newblock \textit{Biometrika} \textbf{43}, 353--360.

\bibitem[{Rajaraman et al.(2017)}]{Raj17}
\textsc{Rajaraman, N., Thangaraj, A. and Suresh, A.T.} (2017)
\newblock {Minimax Risk for Missing Mass Estimation.} 
\newblock \textit{Proceedings of the IEEE International Symposium on Information Theory}.

\bibitem[{Robbins(1968)}]{Rob(68)}
\textsc{Robbins, H.} (1968).
\newblock {Estimating the total probability of the unobserved outcomes of an experiment.}
\newblock \textit{The Annals of Mathematical Statistics} \textbf{39}, 256-257. 

\bibitem[{Teh and G\"or\"ur(2009)}]{Teh(09)} 
\textsc{Teh, Y.W. and G\"or\"ur, D.} (2009). 
\newblock{Indian buffet processes with power--law behavior.}
\newblock \textit{Advances in Neural Information Processing Systems.}

\bibitem[{Tukey(1958)}]{Tuk(58)}
\textsc{Tukey, J.W.} (1958).
\newblock {Bias and confidence in not-quite large samples}
\newblock \textit{The Annals of Mathematical Statistics} \textbf{29}, 614.

\bibitem[{Wood and Griffiths(2007)}]{Woo(07)}
\textsc{Wood, F. and Griffiths, T.L.} (2007).
\newblock {Particle filtering for nonparametric Bayesian matrix factorization.}
\newblock \textit{Advances in Neural Information Processing Systems.}

\bibitem[{Wood et al.(2006)}]{Woo(06)}
\textsc{Wood, F., Griffiths, T.L. and Ghahramani, Z.} (2006).
\newblock {A non-parametric Bayesian method for inferring hidden causes.}
\newblock \textit{22nd Conference in Uncertainty in Artificial Intelligence.}

\bibitem[{Zou et al.(2016)}]{Zou(16)}
\textsc{Zou, J., Valiant, G., Valiant, P., Karczewski, K., Chan, S.O., Samocha, K., Lek, M., Sunyaev, S., Daly, M. and MacArthur, D.G.} (2016).
\newblock {Quantifying the unobserved protein-coding variants in human populations provides a roadmap for large-scale sequencing projects.}
\newblock \textit{Nature Communications} \textbf{7}.


\end{thebibliography}
\end{document}